\documentclass[12pt,a4paper]{article}

\usepackage{graphicx,hyperref,enumitem}

\usepackage{amsmath,amsthm,amssymb,color}
\usepackage[noadjust,sort]{cite}
\usepackage{hyperref}

\usepackage[font=small,labelfont=bf]{caption}
\normalsize

\newtheorem{thm}{Theorem}[section]
\newtheorem{cor}[thm]{Corollary}
\newtheorem{lemma}[thm]{Lemma}
\newtheorem{conj}[thm]{Conjecture}

\numberwithin{equation}{section}

\newcommand{\dmax}{d_{\mathrm{max}}}
\newcommand{\dmin}{d_{\mathrm{min}}}
\newcommand{\din}{d_{\mathrm{in}}}
\newcommand{\dout}{d_{\mathrm{out}}}
\newcommand{\dav}{\bar d}
\newcommand{\dgeom}{\hat d}
\newcommand{\dharm}{d_{\mathrm{hm}}}

\newcommand{\st}{\mathrel{:}}
\let\le=\leqslant
\let\ge=\geqslant

\newcommand{\RT}{\operatorname{RT}}
\newcommand{\EO}{\operatorname{EO}}
\newcommand{\NT}{\operatorname{NT}}

\let\originalleft\left
\let\originalright\right
\renewcommand{\left}{\mathopen{}\mathclose\bgroup\originalleft}
\renewcommand{\right}{\aftergroup\egroup\originalright}

\newcommand{\abs}[1]{\lvert#1\rvert} \let\card=\abs

\renewcommand{\dfrac}[2]{\lower0.12ex\hbox{\large$\textstyle\frac{#1}{#2}$}}
\newcommand{\Dfrac}[2]{\raise0.05ex\hbox{\small$\displaystyle\frac{#1}{#2}$}}
\newcommand{\eps}{\varepsilon}

\newcommand{\calC}{\mathcal{C}}
\newcommand{\calG}{\mathcal{G}}
\newcommand{\calP}{\mathcal{P}}
\newcommand{\thetavec}{\boldsymbol{\theta}}
\newcommand{\phivec}{\boldsymbol{\phi}}
\newcommand{\dvec}{\boldsymbol{d}}
\newcommand{\fvec}{\boldsymbol{f}}
\newcommand{\vvec}{\boldsymbol{v}}
\newcommand{\xvec}{\boldsymbol{x}}
\newcommand{\yvec}{\boldsymbol{y}}
\newcommand{\zvec}{\boldsymbol{z}}

\newcommand{\paul}{\widehat\rho}
\newcommand{\eo}{\rho}
\newcommand{\ours}{\rho_\tau}
\newcommand{\rhonew}{\rho_{\mathrm{new}}}

\newcommand{\Integers}{\mathbb{Z}}
\newcommand{\E}{\operatorname{\mathbb{E}}}
\renewcommand{\Pr}{\operatorname{\mathbb{P}}}

\renewcommand{\P}[1]{\Pr\bigl(#1\bigr)}
\newcommand{\cart}{\mathbin{\raise0.15ex\hbox{$\scriptstyle\square$}}}

\def\nicebreak{\vskip 0pt plus 50pt\penalty-300\vskip 0pt plus -50pt }

\begin{document}

\title{Correlation between residual entropy and spanning tree entropy of ice-type models on graphs}
\author{
Mikhail Isaev\thanks{Supported by Australian Research Council grant DP220103074} \\
\small School of Mathematics and Statistics \\[-0.5ex]
\small UNSW Sydney \\[-0.5ex]
\small Sydney NSW 2052, Australia \\[-0.5ex]
\small\tt m.isaev@unsw.edu.au
\and
Brendan D. McKay\thanks{Supported by Australian Research Council grant DP190100977} \\
\small School of Computing \\[-0.5ex]
\small Australian National University  \\[-0.5ex]
\small Canberra ACT 2601, Australia \\[-0.5ex]
\small\tt brendan.mckay@anu.edu.au
\and
Rui-Ray Zhang \\
\small Simons Laufer Mathematical Sciences Institute \\[-0.5ex]
\small Berkeley CA 94720, USA  \\[-0.5ex]
\small\tt ruizhang@msri.org
}

\date{\today}

\maketitle

\begin{abstract}
The logarithm of the number of Eulerian orientations, normalised by the number of vertices, is known as the residual entropy in studies of ice-type models on graphs.
The spanning tree entropy depends similarly on the number of spanning trees.
We demonstrate and investigate a remarkably strong, though
non-deterministic, correlation between these two entropies.
This leads us to propose a new heuristic estimate for the residual entropy of regular graphs
that performs much better than previous heuristics.
We also study the expansion properties and residual entropy of random graphs with given degrees.
\end{abstract}

\section{Introduction}

The graphs in this paper are free of loops but may
have multiple edges. When it is not clear from the context,
we will use ``simple graph'' or ``multigraph'' to 
emphasise that multiple edges are forbidden or allowed.
An \textit{Eulerian orientation} of a graph is an orientation of its edges such that every vertex has equal in-degree and out-degree.

Let  $G$ be a graph with $n$ vertices and positive even degrees, and
let $\EO(G)$ denote the number of Eulerian orientations of~$G$.
We consider the logarithm of the number of Eulerian orientations of $G$ normalised
by the number of vertices: 
\begin{equation}
\eo(G) := \dfrac{1}{n} \log \EO(G).
\label{def:eo}
\end{equation}
If $G$ is an infinite repeating lattice of bounded degree and
$\{G_i\}$ is an increasing sequence of Eulerian graphs which
are locally like $G$ at most vertices, then under weak conditions
$\eo(G_i)$ converges to a limit~$\eo(G)$ that only depends on~$G$.
See~\cite{Bencs} for a precise definition and proof.
The quantity $\eo(G)$ is known as the \textit{residual entropy} of ice-type models in statistical physics.
Determining the asymptotic behaviour of $\eo(G)$ as $n\to\infty$ is a key question in the area, see for example
\cite[Chapter 8]{baxter1969f}
and \cite{lieb1972two}. 
In particular, the value is known for the square lattice~\cite{lieb1967residual},
the triangular lattice~\cite{baxter1969f}, and the hexagonal ice monolayer~\cite{LieHexagonalMonolayer}. 
In addition, approximate values for many other lattice structures have been proposed, some of which we will mention below.

We can safely ignore graphs with isolated vertices.
Given a degree sequence $\dvec=(d_1,\ldots,d_n)$,
define $\dmin$, $\dav$ and $\dmax$ to be
the minimum, average and maximum degrees.
We will also use the geometric mean $\dgeom=(d_1\cdots d_n)^{1/n}$
and the harmonic mean $\dharm = 
n(d_1^{-1}+\cdots+d_n^{-1})^{-1}$.
Note that $\dmin\le \dharm\le\dgeom\le \dav\le\dmax$.

Around 90 years ago, Pauling \cite{pauling1935structure} proposed the best-known heuristic estimate for~$\eo(G)$.
Orient each edge at random. The probability that any one vertex has in-degree equal to out-degree is
$2^{-d}\binom{d}{d/2}$, where $d$ is the degree of the vertex. Assuming heuristically that these events are independent gives an estimate of $\eo(G)$ that we will call the Pauling estimate:
  \begin{equation}\label{def:Pauling}
    \paul(G) :=
  \Dfrac1n\,\sum_{i=1}^n\,\biggl(\log \binom{d_i}{d_i/2} - \Dfrac{d_i}{2} \log 2\biggr). 
  \end{equation}

Lieb and Wu~\cite{lieb1972two}, and later Schrijver \cite{schrijver1983bounds},
showed that, for any multigraph $G$ with even degrees,  
 \begin{equation}\label{eq:Sch}
    \paul(G) \le \eo(G)
    \le \Dfrac{1}{2n}\sum_{i=1}^n\,\log \binom{d_i}{d_i/2}.
 \end{equation}
Comparing these upper and lower bounds, we find that 
\[
\eo(G) = \paul(G)+ O(\log \dgeom).
\]

In this paper we make a number of theoretical and empirical contributions
to the study of residual entropy.
In Section~\ref{s:results} we first survey other known and conjectured
bounds on $\eo(G)$.
For simple graphs we conjecture that in fact $\eo(G)=\paul(G)+o(1)$
if the harmonic mean degree goes to infinity, which will be the case
if the minimum degree goes to infinity.
By Theorem~\ref{T:IMZ}, this is true for sufficiently dense graphs
with good expansion.
Theorem~\ref{thm:newbound} shows that $\eo(G)=\paul(G)+O(1)$
uniformly for all simple graphs.
In Section~\ref{ss:random}, we report that
$\eo(G) = \paul(G)+o(1)$ for two very broad ranges of random
simple graphs with given degrees.
In Section~\ref{ss:heuristic}, we describe an empirical observation,
supported by Theorem~\ref{T:IMZ} and experiment, that $\eo(G)$ is
highly correlated with the number of spanning trees.
In combination with our knowledge of random graphs, this leads
us to propose a new heuristic $\ours(G)$ which is a much better
estimate of $\eo(G)$ than is
$\paul(G)$ for all the graphs we have tested.

Proofs of the theorems in Section~\ref{s:results} are given in
Sections~\ref{s:newbound} and~\ref{S:randomproof}.
In Section~\ref{s:numeric}, we explain how we computed good
estimates of $\eo(G)$ even for graphs with thousands of vertices.
In Section~\ref{s:products} we demonstrate how our new heuristic
provides good estimates for products of cycles, and propose a
value for the residual entropy of the simple cubic lattice.

In Section~\ref{s:cliques} we use the product of a cycle and
a clique as a test case to explore the residual entropy when
the degree grows as a function of the number of vertices.
We find that $\eo(G) = \paul(G)+o(1)$ in that case, but that
$\ours(G)$ is an even better match for $\eo(G)$ with a
significantly smaller error term.

In Section~\ref{s:others}, we show that our heuristic compares
very favourably with the correct values for the triangular
lattice, two types of 3-dimensional ice, and high-dimensional
hypercubes.

\section{Statements of the main results and conjectures}\label{s:results}

For connected regular multigraphs of degree $d\ge 4$, Las Vergnas \cite[Theorem 4]{LasVergnas1990}
obtained a slightly better lower bound than~\eqref{eq:Sch} and,
on condition of connectivity, a significantly better upper bound.
\begin{equation}\label{eq:LVupper}
  \EO(G) \le K^{2/(d-2)}
  \bigl(2^{d/2g}K^{1-d/g(d-2)}\bigr)^{\!n},
\end{equation}
where $K=2^{-d/2}\binom{d}{d/2}$ and $g$ is the girth.
This bound implies that uniformly
\begin{equation}\label{eq:LVsimple}
    \eo(G) \le \paul(G) + \Dfrac{\log d}{2g} + O(n^{-1}).
\end{equation}

Prior to Las Vergnas' work, a much stronger upper bound for the case of simple
graphs was conjectured by Schrijver.
\begin{conj}[\cite{schrijver1983bounds}]\label{Sconj}
 If a simple graph has even degrees $d_1,\ldots,d_n$, then
 \[ 
\EO(G) \le \prod_{i=1}^n \,\RT(d_i+1)^{1/(d_i+1)},
 \]
 where $\RT(d+1)$ is the number of Eulerian orientations
 of the complete graph (i.e., regular tournaments) with $d+1$ vertices.
\end{conj}

If Conjecture~\ref{Sconj} is correct, it is best possible
for many degree sequences since it is exact for the disjoint union
of complete graphs.
We have computationally confirmed that no other graphs up to 12
vertices achieve the bound, and similarly for 13-vertex graphs with
degrees~4 and~6, 4-regular graphs
up to 19 vertices and 6-regular graphs up to 14 vertices.
From \cite[Theorem 5.1]{IMZ} we know that
\begin{equation}\label{eq:RTvalue}
\RT(d+1) =  d^{1/2} \biggl( \Dfrac{2^{d+2}}{\pi (d+1)}\biggr)^{\!d/2}
e^{-\frac12 + O(d^{-1})}, 
\end{equation}
which enables us to 
calculate (since we disallow isolated vertices) that
Conjecture~\ref{Sconj} would imply that
\[
  \eo(G)\le\paul(G) 
  + \Dfrac1n\sum_{i=1}^n \,\Dfrac{O(1)+\log d_i}{d_i}
  \le \paul(G) + \Dfrac{O(1)+\log\dharm}{\dharm}
\]
for simple graphs, where the second inequality follows by the
concavity of the function $x\log(1/x)$ for $x\ge 0$.

Though we don't know how to prove Conjecture~\ref{Sconj},
our evidence suggests that at least the following
implication is true.

\begin{conj}\label{Conj}
    If $G=G(n)$ is a sequence of simple graphs with even degrees $\dvec=\dvec(n)$,
    such that the harmonic mean degree $\dharm\to\infty$ as $n\to\infty$, then 
    \[
            \eo(G) = \paul(G) +o(1).
    \]
\end{conj}

Recall that $\dmin\le\dharm\le \dgeom\le\dav\le\dmax$. 
The condition $\dharm\to\infty$ is implied by $\dmin\to\infty$,
but the weaker condition $\dgeom\to\infty$ is not sufficient.
For $n$ being an odd multiple of~6, define $G_n$ to
be the disjoint union of $K_{n/2}$ and $n/6$ triangles.
Then the geometric mean degree is $\dgeom\sim\sqrt n$, but
$\EO(G_n)=2^{n/6} \RT(n/2)$, which implies by~\eqref{eq:RTvalue}
that $\eo(G_n)-\paul(G_n)\to \frac16\log 2$.
Note that the harmonic mean degree is $4+o(1)$.

The converse of Conjecture~\ref{Conj} is also not true.
Even $\dmax=O(1)$ is insufficient to imply
$\eo(G)=\paul(G)+\Omega(1)$, as shown by the case of
increasing girth in~\eqref{eq:LVsimple}.
Another observation is that
if $G_1,G_2,\ldots$ is any sequence of graphs such that
$\eo(G_i)=\paul(G_i)+o(1)$, then the same is true for
$G'_1,G'_2,\ldots$, where $G'_i$ is $G_i$ with any
number of edges subdivided, even though the
average degree may approach~2.

It suffices to prove Conjecture~\ref{Conj}
for connected simple graphs, on account of the following
lemma whose proof will appear in Section~\ref{s:newbound}.

\begin{lemma}\label{lem:discon}
  Let $\calC$ be a class of connected simple graphs for which
  Conjecture~\ref{Conj} holds.
  Then the conjecture also holds for graphs whose
  components all lie in~$\calC$.
\end{lemma}


\subsection{Sufficiently growing degrees and good expansion}

A recent result of the present authors shows that Conjecture~\ref{Conj} holds
for simple graphs with good expansion properties and sufficiently high degrees. Recall that the \textit{isoperimetric number} (also known as the Cheeger constant) of a graph $G$ is defined by
\begin{equation}
	h(G):=\min
	\biggl\{ \Dfrac{|\partial_G \,U|}{|U|}  \st  U\subset V(G), 1\le |U|\le\dfrac12 |V(G)| \biggr\},
 \label{cheeger}
\end{equation}
where $\partial_G \,U$ is the set of edges of $G$ with one end in $U$ 
and the other end in $V(G) \setminus U$.
Note that $h(G)\le \dmin(G)$. 

\begin{thm}[Isaev, McKay, Zhang \cite{IMZ}] \label{T:IMZ}
Let $G = G(n)$ be a simple graph with $n$ vertices and
even degrees.
Assume that  
 $\dmax \gg \log^{8} n$ and   $h(G) \ge \gamma \dmax$ for some constant $\gamma > 0$. Then,  
    \[
        \EO(G) = \Dfrac{ 2^{ |E(G)| } }{ \sqrt{ t(G) } }
\,\biggl( \Dfrac{2}{\pi} \biggr)^{(n-1)/2} \!\!
\exp\biggl( -\dfrac14 \sum_{ jk \in G } 
\biggl( \Dfrac{1}{d_{j}} + \Dfrac{1}{d_{k}} \biggr)^{\!2\,} + O\biggl(\Dfrac{n}{\dmax^{2}} \log\Dfrac{2n}{\dmax} \biggr)\biggr),
    \]
    where 
$t(G)$ is the number of spanning trees of $G$.
\end{thm}

In relation to Conjecture \ref{Conj}, the following  consequence of Theorem \ref{T:IMZ}
 is proved in Section~\ref{s:newbound}.

\begin{cor}\label{cor:eo}
    Under the assumptions of Theorem \ref{T:IMZ}, we have that 
    \[
       \eo(G) = \paul(G) +  O\left( \Dfrac{\log^2 \dmin} {\dmin}\right).
    \]
\end{cor}

\subsection{A new upper bound on $\EO(G)$}

Our next contribution is a new upper bound on~$\eo(G)$,
which implies that $\eo(G)=\paul(G)+O(1)$ for all simple graphs.

\begin{thm}\label{thm:newbound}
For any   connected multigraph $G$ with even degrees 
and $t(G)$ spanning trees,  
\[
    \EO(G) \le \Dfrac{2^{\card{E(G)}+3(n-1)/2}}{\pi^{(n-1)/2}\,t(G)^{1/2}}.
\]
\end{thm}

\begin{cor}\label{C:bounded}
 For simple graphs~$G$ with even degrees,
 \[
    \paul(G) \le \eo(G) \le \paul(G) + \begin{cases}
             \,\dfrac{27}{10},&\text{~always;} \\[0.5ex]
             \,\dfrac{21}{22},&\text{~if $G$ is regular;} \\[0.5ex]
             \,\dfrac{5}{6},&\text{~if $G$ is regular without 3-cycles.}
      \end{cases}
 \]
 Moreover, if $\dmin=\dmin(n)\to\infty$,
  \[
    \paul(G) \le \eo(G) \le \paul(G) + \log 2 
            + O\Bigl(\Dfrac {\log^2 \dmin}{\dmin}\Bigr).
 \]
\end{cor}

The above theorem and its corollary are proved in Section~\ref{s:newbound}.
If Conjecture~\ref{Sconj} is correct, the largest value
of $\eo(G)-\paul(G)$ for simple graphs is actually
$\frac13\log 2\approx 0.23105$, occurring for disjoint
unions of triangles.
Note that, although Theorem~\ref{thm:newbound} holds for multigraphs, Corollary~\ref{C:bounded} does not.
For multigraphs, the upper bound in~\eqref{eq:Sch}
is achieved, so $\eo(G)-\paul(G)$ is not uniformly bounded.

\subsection{Random graphs with given degrees}\label{ss:random}

Let $\calG(n,\dvec)$ denote the uniform probability space of simple graphs with $n$ vertices and
degree sequence~$\dvec$.
We will prove that Conjecture~\ref{Conj} holds in the following probabilistic sense. 
The proof will appear in Section~\ref{S:randomproof}.

 \begin{thm}\label{T:main}
   Let $\dvec=(d_1,\ldots,d_n)$ be a graphical degree sequence such that each $d_i$ is positive and even
  and either of the following two conditions holds:
  \begin{itemize}\itemsep=0pt
    \item[(R1)] $\dmax^2 = o(n)$,
    \item[(R2)] $\dmax \gg \log^8n$ and $\dmin \ge \gamma \dmax$ for some fixed $\gamma>0$.
  \end{itemize}
   If  $G \sim \calG(n,\dvec)$  then,
   for any fixed $\eps>0$, 
   \[
        \P{\eo(G) > \paul(G)+\eps} \le e^{-\Omega(\dmax^2+ n)}.
   \]
 \end{thm}
 Theorem \ref{T:main} 
 follows immediately from two more detailed results Theorem \ref{T:sparse}  and Theorem \ref{T:dense} that explore the dependence of the probability bounds with respect to $\eps$. In particular, these results show that if $G \sim \calG(n,\dvec)$ then, with probability tending to~$1$,
 \begin{itemize}
 \item  $ \eo(G) \le  \paul(G)+ O\left(\dfrac{\dmax^2+\log n}{n}\right)$
 for the range (R1) of Theorem \ref{T:main}; 
 \item   $ \eo(G) \le  \paul(G)+ O\left(\dfrac{\log^2 \dmax}{\dmax}\right)$
 for the  range (R2) of Theorem \ref{T:main}.
 \end{itemize}

For the range (R1), we combine the result of \cite{silver} on the enumeration of bipartite graphs with the switching method to find 
an asymptotic formula for 
\[
  \E\,[\EO(G)]=\E\, [e^{n \eo (G)}].
\]
It turns out that  
\[ 
   \E\,[e^{n \eo (G)}] = e^{n \paul(G) + O(\dmax^2)}.
\]
Then we use standard arguments to show the concentration of $\eo(G)$.

 For the range (R2), we employ Theorem \ref{T:IMZ}. 
 Note that it requires $G$ to have a sufficiently large isoperimetric constant. Applying the switching method, we show that random graphs with given degrees are good expanders with high probability, which could be of independent interest; see Section \ref{S:same-order}.

\subsection{A new heuristic estimate for regular graphs}\label{ss:heuristic}

The asymptotic formula in Theorem \ref{T:IMZ} suggests that $\eo(G) + \frac{1}{2} \tau(G)$  
may exhibit much less dependency than $\eo(G)$ on the structure of the graph.
Here the \textit{spanning tree entropy} $\tau(G)$ is the logarithm of the number of spanning trees of $G$ normalised by the number of vertices: 
\[
  \tau(G):= \dfrac{1}{n} \log t(G).
\]

As a consequence of  Theorem  \ref{T:main} we have established that 
Pauling's estimate $\paul(G) = \log\binom{d}{d/2} - \dfrac{d}{2} \log 2$ 
is asymptotically correct for random simple $d$-regular graphs
provided $d(n)$ grows quickly enough. 
McKay \cite{mckay1983spanning}  showed that   spanning tree entropy for this
random graph model (for the case when $3\le d= O(1)$) is concentrated around
\[
    \tau_{d} := \log\Dfrac{(d-1)^{d-1}}{(d^2-2d)^{d/2-1}}.
\]
Thus, it is natural to consider the following estimate for the residual entropy of a simple $d$-regular graph based on a correction of the number of spanning trees with respect to a random graph:
\begin{equation}\label{ours}
\ours(G) := \paul(G) + \dfrac12  \tau_{d} - \dfrac{1}{2} \tau(G).
\end{equation}
It follows from \cite[Thm.\ 5.2]{mckay1983spanning} that
$\tau(G)<\tau_d(G)$ if $G$ is $d$-regular for $d\ge 4$, except
possibly for $d=4, n\le 18$.
Thus, our estimate is consistent with $\eo(G)\ge\paul(G)$.

\begin{figure}[ht]
  \centering
\unitlength=1cm
  \begin{picture}(15,7.5)(0,0)
    \put(0.2,0.5){\includegraphics[scale=0.35]{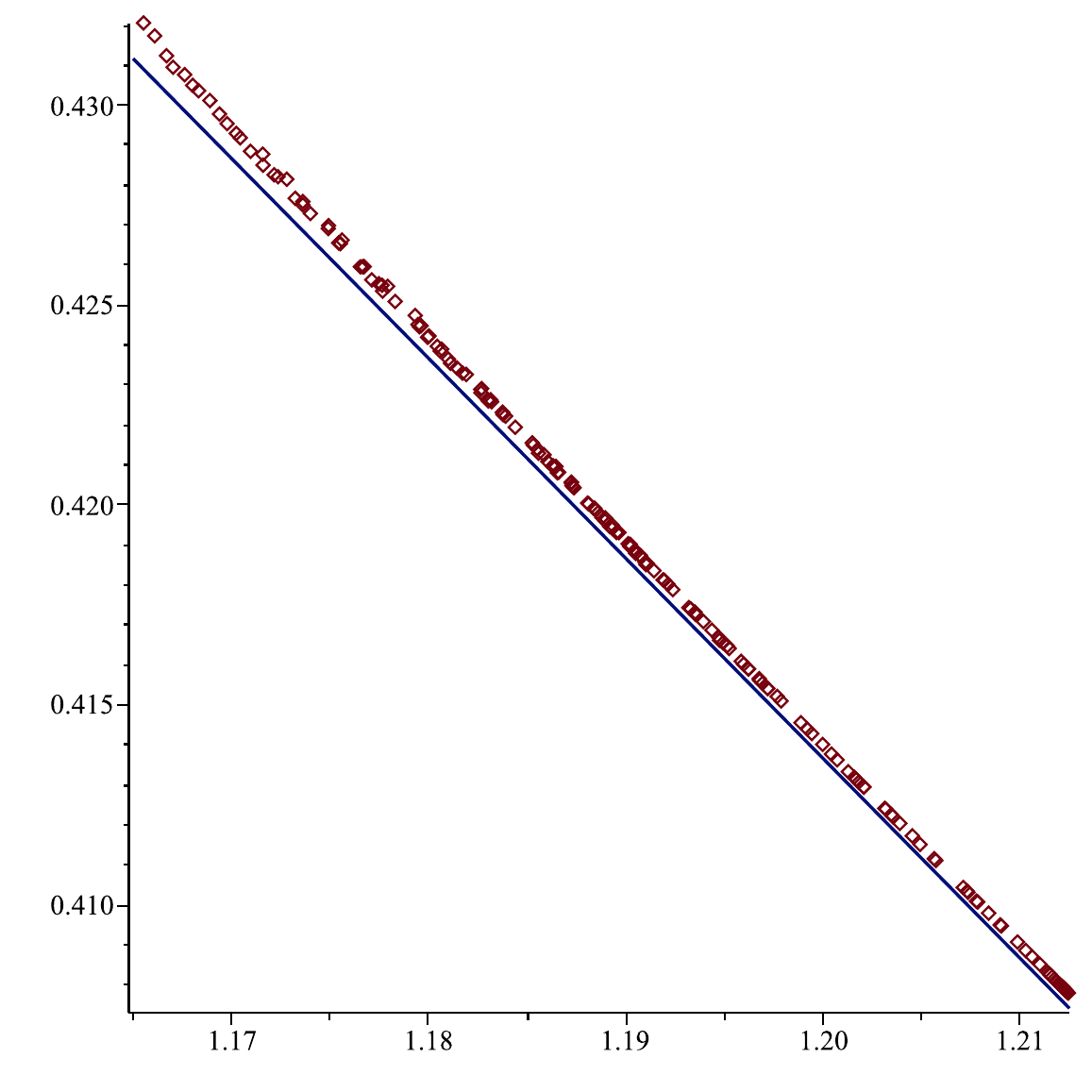}}
    \put(8,0.5){\includegraphics[scale=0.35]{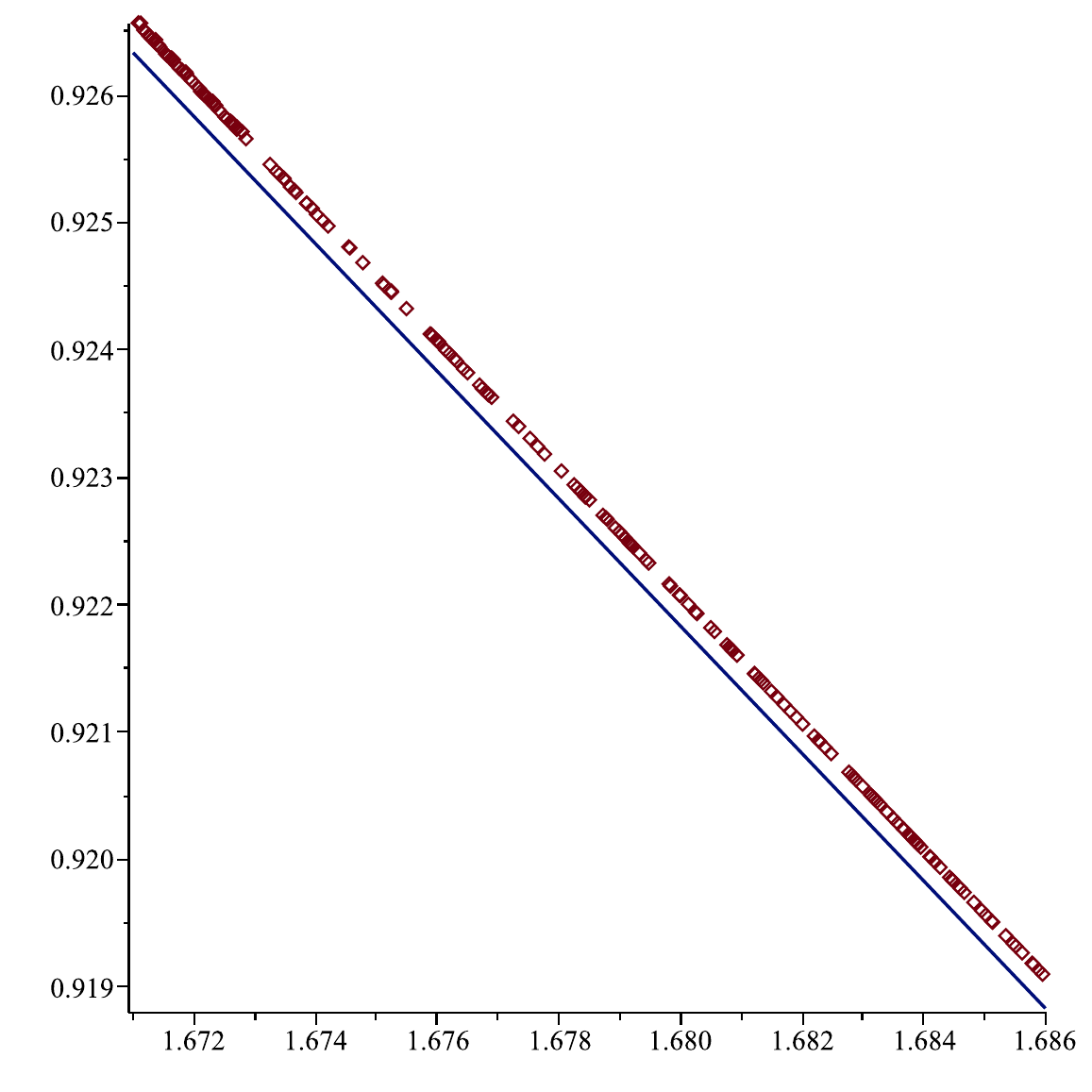}}
    \put(1.9,0.1){corrupted $C_{40}\cart C_{40}$}
    \put(9.3,0.1){corrupted $C_{12}\cart C_{12}\cart C_{12}$}
  \end{picture}
 \caption{$\eo(G)$ (vertical) versus $\tau(G)$ (horizontal) for some randomised graphs. The solid line is $\ours(G)$.\label{fig:randomish}}
\end{figure}

In order to illustrate our case for this estimate, we tested
several hundred large graphs of degree 4 or~6.
The method by which $\eo(G)$ was estimated will be
described in Section~\ref{s:numeric}.
To show a continuum between a regular lattice structure
and a random graph, we started with a 2-dimensional
square lattice $C_{40}\cart C_{40}$ (1600 vertices,
degree~4), and a 3-dimensional simple cubic lattice
$C_{12}\cart C_{12}\cart C_{12}$ (1728 vertices, degree~6)
and applied the switching $\{ab,cd\}\to\{ac,cd\}$ in
random places between 0 and 10,000 times (avoiding multiple edges and loops).
The results shown in Figure~\ref{fig:randomish} suggest
that the correlation between $\eo(G)$ and $\tau(G)$
is even stronger than between $\eo(G)$ and $\ours(G)$.
However, the  relationship between the two entropies is
not exact, as shown in Figure~\ref{fig:samest}.

\begin{figure}[ht]
  \[ \includegraphics[scale=0.5]{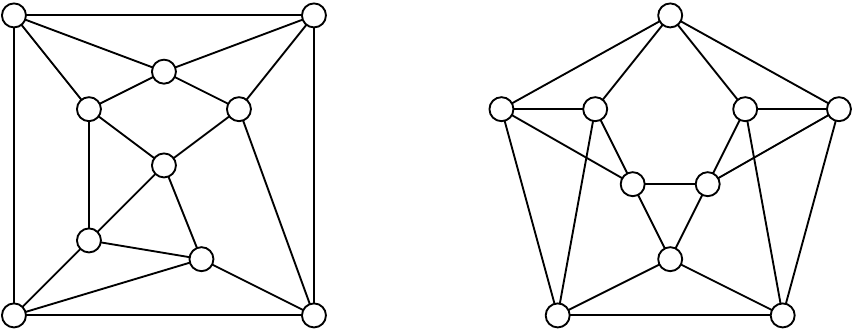} \]
 \caption{Two 4-regular graphs with the same 
 eigenvalues and number of spanning trees but different numbers of Eulerian orientations\label{fig:samest}}
\end{figure}
The next theorem  shows that even a weaker asymptotic relationship fails to hold in general.
This is unfortunate as such a relationship would be sufficient to demonstrate that the residual
entropies of hexagonal ice (Ih) and cubic ice (Ic) are identical; see Section \ref{S:3Dice}.
\begin{thm}
  There is no function $\mathring\eo(d,\tau)$ with the following property
  for every~$d$ and~$\tau$.
  If $G_1,G_2,\ldots$ is an increasing sequence of connected
  $d$-regular graphs such that $\tau(G_i)\to\tau$,
  then $\eo(G_i)\to\mathring\eo(d,\tau)$.
\end{thm}

\begin{proof}
  (See Section~\ref{s:products} for the definitions and 
  elementary theory.)
  For $m\ge 3$, define the Cartesian products $H_1(m):=G_1\cart C_m$
  and $H_2(m) := G_2\cart C_m$, where $G_1$ and $G_2$ are the graphs in Figure~\ref{fig:samest} and $C_m$ is the cycle of length~$m$.
  Since $H_1(m)$ and $H_2(m)$ have the same eigenvalues, they have the same spanning tree entropy and it converges to a limit as $m\to\infty$.
  However, using the transfer matrix method
  (see Theorem~\ref{thm:transfer} below),
  we have determined that
  \[
     \lim_{m\to\infty}\eo(H_1(m)) \approx 0.9525279,
     \qquad
     \lim_{m\to\infty}\eo(H_2(m)) \approx 0.9524817.
  \]
  Since the two limits are different, $\mathring\eo(6,\tau)$ doesn't exist.
\end{proof}

\nicebreak
\section{Proof of Theorem \ref{thm:newbound},
Corollary~\ref{cor:eo} and Lemma~\ref{lem:discon}}\label{s:newbound}

The \textit{Laplacian matrix} $L(G)$ of a loop-free multigraph $G$ with $n$ vertices
is the $n \times n$ matrix with entries given by
\[
L_{ij} = 
\begin{cases}
\,d_i, & \text{ if $i = j$ and $d_i$ is the degree of vertex $i$}; \\
\,-\ell, & \text{ if $i\neq j$ and $\ell$ edges join vertices $i$ and $j$}.
\end{cases}
\]
As is well known, the number of zero eigenvalues of $L(G)$ equals the number of components of~$G$, and the other eigenvalues are strictly positive.
In addition, the Matrix Tree Theorem says that the number of spanning trees of $G$ is the absolute value of every minor of~$L(G)$.

\begin{proof}[Proof of Theorem~\ref{thm:newbound}.]
 The number of Eulerian orientations of $G$ is the constant term in
 \[
      \prod_{jk\in G}\, \Bigl( \Dfrac{x_j}{x_k} + \Dfrac{x_k}{x_j}\Bigr).
 \]
  As shown in \cite{Tejas}, this implies that
   \[
         \EO(G) = 2^{\card{E(G)}} \pi^{-n} \int_{-\pi/2}^{\pi/2}\!\!\! \cdots\!\int_{-\pi/2}^{\pi/2}
         F(\thetavec)\,d\thetavec,
   \]
   where $\thetavec=(\theta_1,\ldots,\theta_n)$ and
   \[
       F(\thetavec) = \prod_{jk\in G} \cos(\theta_j-\theta_k).
   \]
   Since $F(\thetavec)$ is invariant under uniform shift of the arguments, we
   can fix $\theta_n=0$ and write
    \[
         \EO(G) = 2^{\card{E(G)}} \pi^{-n+1} \int_{-\pi/2}^{\pi/2}\!\!\! \cdots\!\int_{-\pi/2}^{\pi/2}
         F(\theta_1,\ldots,\theta_{n-1},0)\,d\theta_1\cdots d\theta_{n-1}.
   \]
   Define $\phivec=(\phi_1,\ldots,\phi_{n-1})$, where
   \[
        \phi_j = \begin{cases}
                          \, \frac\pi2 - \theta_j, &~\text{ if $\frac\pi4\le\theta_j\le\frac\pi2$}; \\
                           \;\theta_j,                 &~\text{ if $-\frac\pi4\le\theta_j\le\frac\pi4$}; \\
                           \,-\frac\pi2 - \theta_j, &~\text{ if $-\frac\pi2\le\theta_j\le-\frac\pi4$},
                        \end{cases}
   \]
   and also set $\phi_n=0$ for convenience.
   Note that $\thetavec\in\bigl[-\frac\pi2,\frac\pi2\bigr]^n$ implies
   $\phivec\in\bigl[-\frac\pi4,\frac\pi4\bigr]^n$. 
   By considering all the cases, we can check that, for $\theta_j,\theta_k\in 
   \bigl[-\frac\pi2,\frac\pi2\bigr]$,
   \[
        \abs{\phi_j-\phi_k} \le \min\bigl\{ \abs{\theta_j-\theta_k},\pi-\abs{\theta_j-\theta_k}\bigr\}.
   \]
   We also have, for $x\in[-\pi,\pi]$, that $\abs{\cos x}\le\exp\bigl(-\frac12\min\{\abs x,\pi-\abs x\}^2\bigr)$.
   Therefore, for $\thetavec\in\bigl[-\frac\pi2,\frac\pi2\bigr]^n$, we have
   \[
        \abs{F(\thetavec)} \le \exp\Bigl(-\dfrac12 \sum_{jk\in G} (\phi_j-\phi_k)^2 \Bigr)
        = \exp\Bigl( -\dfrac12 \phivec^{\mathrm T\!}Q\phivec\Bigr),
   \]
   where $Q=Q(G)$ is the Laplacian matrix of~$G$ with the final row and column removed.
   In changing variables from $\thetavec$ to $\phivec$ in the integral, we need to take account
   of the fact that $\theta_j\mapsto\phi_j$ is a two-to-one map.  To compensate for this,
   we multiply by $2^{n-1}$:
    \begin{align*}
         \EO(G) &\le 2^{\card{E(G)}+n-1} \pi^{-n+1} \int_{-\pi/4}^{\pi/4}\!\! \cdots\!\int_{-\pi/4}^{\pi/4}
          \exp\Bigl( -\dfrac12 \phivec^{\mathrm T\!}Q\phivec\Bigr)\,d\phivec \\
          &\le 2^{\card{E(G)}+n-1} \pi^{-n+1} \int_{-\infty}^{\infty}\!\! \cdots\!\int_{-\infty}^{\infty}
          \exp\Bigl( -\dfrac12 \phivec^{\mathrm T\!}Q\phivec\Bigr)\,d\phivec \\
          &= 2^{\card{E(G)}+3(n-1)/2} \pi^{-(n-1)/2} \abs{Q}^{-1/2} \\
          &= 2^{\card{E(G)}+3(n-1)/2} \pi^{-(n-1)/2} t(G)^{-1/2},
   \end{align*}
   where the final step is to apply the Matrix Tree Theorem.
\end{proof}

\begin{lemma}[Kostochka~\cite{Kostochka}]\label{lem:kost}
A connected simple graph $G$ with $n$
vertices and minimum degree $\dmin\ge 2$ has
\[
     \bigl(\dgeom\, \dmin^{-6\log\dmin/\dmin}M(\dmin)^{-1}\bigr)^n
     \le t(G) \le \Dfrac{1}{n-1}\dgeom^{\,n},
\]
where $\dgeom$ is the geometric mean degree
and $M(\dmin)=\min\{ 2, \dmin^{3\log \dmin/\dmin} \}$.
\end{lemma}
\begin{proof}
 This is not stated explicitly in~\cite{Kostochka} but follows
 from the proof given there plus the observation that
 \[
   \sum_{i=0}^{\lfloor 3n\log k/k\rfloor} \binom{n-1}{i}\le
   \dfrac12 \min\{ 2, k^{3\log k/k} \}^n
 \]
 for $2\le k\le n-1$.
 Note that the constants in Kostochka's proof are far from
 optimized. Improving them would also improve the constants
 in our Corollary~\ref{C:bounded}, but we have not attempted
 to do that.
\end{proof}

\begin{proof}[Proof of Corollary~\ref{C:bounded}]
Suppose $G$ is disconnected, with components $G_1,\ldots,G_m$.
Let $n_i$ be the number of vertices of $G_i$,
and let $n=n_1+\cdots+n_m$.
From the definition of $\paul(G)$, and the fact
that $\EO(G)=\prod_{i=1}^m \EO(G_i)$, we have
\begin{equation}\label{eq:disbound}
      \eo(G)-\paul(G) = \sum_{i=1}^m\, \Dfrac{n_i}{n} 
      \bigl(\eo(G_i)-\paul(G_i)\bigr)
      \le \max\nolimits_{i=1}^m \bigl(\eo(G_i)-\paul(G_i)\bigr).
\end{equation}
Therefore it suffices to prove Corollary~\ref{C:bounded} for
connected graphs~$G$.

For $d\ge 2$, asymptotic expansion for large $d$ and
computation for small $d$ gives
\begin{equation}\label{binombound}
    \log \binom{d}{d/2} \ge  d \log 2 - \dfrac12 \log d - \dfrac12 \log \Dfrac{\pi}{2} - 
    \Dfrac{1}{4d}.
\end{equation}
     
 Define $\rhonew(G)$ to be the upper bound on $\eo(G)$
 given by Theorem~\ref{thm:newbound} in conjunction with
 Lemma~\ref{lem:kost}.
 Then
 \begin{align}
     \rhonew(G)-\eo(G) &\le \Dfrac1n \sum_{i=0}^n \delta(d_i,\dmin,n)
     \le \max\nolimits_{i=1}^n \delta(d_i,\dmin,n),
            \quad\text{where} \notag \\
     \delta(d_i,\dmin,n) 
     &= d_i\log 2 + 2\log 2 - \dfrac12\log\pi
     - \dfrac12\log d_i - \log\binom{d_i}{d_i/2} \notag \\[-0.5ex]
     &{\qquad}+\Dfrac{3}{\dmin}\log^2\dmin
     + \Dfrac{1}{2n}\log\Dfrac\pi 8 \label{eq:deltabnd1}\\
     &\le \dfrac32\log 2 + \Dfrac{1}{4\dmin} 
        + \Dfrac{3}{\dmin}\log^2\dmin,\label{eq:deltabnd2}
 \end{align}
where we have applied $M(\dmin)\le 2$, \eqref{binombound},
$\log\frac{\pi}{8}<0$ and $d_i\ge \dmin$.
The expression~\eqref{eq:deltabnd2} is unimodal, with its largest
value for even $\dmin$ at $\dmin=8$ and other values less than 2.69.
The value of~\eqref{eq:deltabnd1} for $d_i=\dmin=8$ and
$n\to\infty$ is less than $2.69242<\frac{27}{10}$.

In the case of regular graphs, the second claim is trivial if
the degree $d=2$ so assume $d\ge 4$.
For $4\le d\le 760$ in the second case, and
$4\le d\le 1900$ in the third case, the bounds follow
from~\eqref{eq:LVupper}, noting that the worst case
for $n$ is $n=d+1$.
For larger degrees, apply 
Theorem~\ref{thm:newbound} and Lemma~\ref{lem:kost} with $M(\dmin)\le \dmin^{3\log\dmin/\dmin}$.

The last claim follows from \eqref{eq:deltabnd2}.
\end{proof}

\begin{proof}[Proof of Lemma~\ref{lem:discon}]
Suppose $G$ has components $G_1,\ldots,G_m$ with orders 
$\alpha_1 n,\ldots,\alpha_m n$.
Define 
\[
     g(h)=\sup \bigl\{\eo(H)-\paul(H) \st \dharm(H)\ge h, H\in\calC\bigr\}.
\]
Clearly $g(h)$ is non-increasing, and by the
assumptions of the lemma, $g(h)\to 0$ as $h\to\infty$.
Suppose $\dharm(G)\to\infty$.
Without loss of generality, for some $\ell\ge 0$,
$\dharm(G_i) \le\dharm(G)^{1/2}$ for $1\le i\le\ell$ and
$\dharm(G_i) >\dharm(G)^{1/2}$ for $\ell+1\le i\le m$.
By the definition of harmonic mean,
$\dharm(G)^{-1} = \sum_{i=1}^m \alpha_i\dharm(G_i)^{-1}
\ge \sum_{i=1}^\ell \alpha_i\dharm(G)^{-1/2}$, so
$\sum_{i=1}^\ell \alpha_i\le \dharm(G)^{-1/2}$.
Now, by~\eqref{eq:disbound} and Corollary~\ref{C:bounded},
\[
  \eo(G)-\paul(G) = \sum_{i=1}^m \alpha_i (\eo(G_i)-\paul(G_i)) 
   \le \dfrac{27}{10} \sum_{i=1}^\ell \alpha_i
   + \sum_{i=\ell+1}^m \alpha_i g(\dharm(G)^{1/2}) = o(1).
   \qedhere
\]
\end{proof}

\begin{proof}[Proof of Corollary~\ref{cor:eo}]
Observe that 
 definition \eqref{cheeger} and assumption $h(G)\ge \gamma \dmax$ imply
 that $\dmin \ge \gamma \dmax$. Applying Theorem \ref{T:IMZ} and
 using Lemma~\ref{lem:kost} and the asymptotic formula~\eqref{binombound}
 we obtain the claimed bound.
\end{proof}

\nicebreak
\section{Proof of Theorem \ref{T:main}}\label{S:randomproof}

We first prove the case of $\dmax=o(\sqrt{n})$. 

\begin{lemma}\label{sparserandom}
  Let $\dvec=\dvec(n) =(d_1,\ldots,d_n)$ be a degree sequence with all degrees positive and even.
  Assume $\dmax^2 = o(\dav n)$.
  Then the average number of Eulerian orientations of a random undirected
  simple graph with degree sequence $\dvec$ is
\[
     n^{1/2}\, 2^{-\dav n/2}
     e^{O(\dmax^2) }
     \prod_{i=1}^{n} \binom{d_i}{d_i/2}.
\]
\end{lemma}
\begin{proof}
The average is equal to the number of Eulerian oriented graphs
with in-degrees and out-degrees $\frac12\dvec$, divided by the number
of undirected graphs with degrees~$\dvec$.

A digraph with in-degrees and out-degrees $\frac12\dvec$ can be
represented as an undirected bipartite graph with vertices
$v_1,\ldots,v_n$ and $w_1,\ldots,w_n$. For each~$i$, both $v_i$ and~$w_i$
have degree~$d_i/2$. Vertex~$v_i$ is adjacent to~$w_j$ if
the digraph has an edge from vertex~$i$ to vertex~$j$.
Since the digraph has no loops, $v_i$ is not adjacent to~$w_i$ for
any~$i$.  With such restrictions, the number of bipartite graphs
given in~\cite[Theorem~4.6]{silver} is, to the precision we need here,
\begin{equation}\label{digcount}
    \Dfrac{(\dav n/2)!}{\bigl(\prod_{i=1}^n (d_i/2)!\bigr)^2}\,
                   e^{O(\dmax^2) }.
\end{equation}
However, such bipartite graphs may correspond to digraphs with 2-cycles,
which is not permitted for simple Eulerian oriented graphs.
This occurs if, for some $i\ne j$, $v_i$ is adjacent to $w_j$,
and $v_j$ is adjacent to $w_i$.
Arbitrarily order all such potential bad pairs and label the corresponding 
events by $E_1,\ldots,E_{\binom n2}$. The probability that none of
these events occurs is the product of conditional probabilities:
\begin{equation}\label{Eventprod}
    P_{\dvec} = \prod_{i=1}^{\binom n2}\,
        \bigl(1 - \Pr( E_i \mid \cap_{1\le j<i} \bar E_j) \bigr),
\end{equation}
where $\bar E_j$ is the complement of~$E_j$.
Consider the bad event $E_x$ that $\{v_i,w_j\}$ and $\{v_j,w_i\}$ are
both edges. Choose two additional edges $\{v',w'\}$ and $\{v'',w''\}$,
and replace these four edges with $\{v_i,w'\}$, $\{v_j,w''\}$,
$\{v',w_i\}$ and $\{v'',w_j\}$. Provided $v'\ne v''$, $w'\ne w''$
and none of the additional edges are incident to any of
$v_i,v_j,w_i,w_j$ or their neighbours, this operation removes the
bad pair of edges $\{v_i,w_j\}$ and $\{v_j,w_i\}$
without creating any additional bad pairs.
Observe that it can
be performed in $\frac14 n^2\dav^2 - O(n\dmax^2\dav)
=\Omega(n^2\dav^2)$ ways
and doesn't change the degree sequence.

In the reverse direction, if the event $E_x$ does not occur and we
wish to create it, the operation is determined by the choice of
an edge incident with each of $v_i,v_j,w_i,w_j$, which can be made
in $O(\dmax^4)$ ways.
Note that both these bounds hold 
irrespective of any conditioning on other bad events not occurring.
Therefore, the conditional probability
of this bad event occurring is $O\bigl(\dmax^4/(n^2\dav^2)\bigr)$.
By~\eqref{Eventprod},
\[
    P_{\dvec} = \bigl( 1 - O(\dmax^4/(n^2\dav^2))\bigr)^{\binom n2}
       =   e^{-O(\dmax^2) }.
\]
Thus we find that the number of Eulerian oriented graphs with in-degrees
and out-degrees~$\frac12\dvec$ is also given by~\eqref{digcount}.

The number of undirected simple graphs with degrees~$\dvec$ was
determined in \cite[Theorem~4.6]{symmetric} to be
\begin{equation}\label{unlabelledcount}
   \Dfrac{(\dav n)!}{(\dav n/2)!\, 2^{\dav n/2} \prod_{i=1}^n d_i! }\,
         e^{O(\dmax^2) }.
\end{equation}
Dividing~\eqref{digcount} by~\eqref{unlabelledcount} completes the proof.
\end{proof}

\begin{thm}\label{T:sparse}
Consider $G\sim\calG(n,\dvec)$, where $\dvec$ satisfies the assumptions of Lemma~\ref{sparserandom}.  Then, for any $\mu=\mu(n)> 0$,
\[
 \P{\eo(G) \ge \paul(G) + \mu} \le \dfrac{1}{{e^{n\mu}-1}} n^{1/2}e^{O(\dmax^2)}.
\]
\end{thm}
\begin{proof}
 Let $p=\P{\eo(G) \ge \paul(G) + \mu}$.
 By \eqref{eq:Sch}, we have $\eo(G)\ge\paul(G)$ always.
 Then
 \[
    \E\,[ \EO(G)]
    = \E\,[ e^{n \eo(G)} ] 
    \ge (1-p)e^{n\paul(G)}+p e^{n(\paul(G)+\mu)}
    = (1-p+e^{n\mu}p)\,e^{n\paul(G)}.
 \]
 Comparing this to the estimate of the expectation
 in Lemma~\ref{sparserandom}, we have
\[
(1-p+e^{n\mu}p)\,e^{n\paul(G)}
\le 
n^{1/2}\, 2^{-\dav n/2}
     e^{ O(\dmax^2) }
     \prod_{i=1}^n\binom{d_i}{d_i/2},
\]
and by the definition of $\paul(G)$ in \eqref{def:Pauling},
\[
p \le \Dfrac{1}{{e^{n\mu}-1}}
\biggl(n^{1/2}\, 2^{-\dav n/2} e^{ O(\dmax^2) }
\prod_{i=1}^n \binom{d_i}{d_i/2}
- 2^{-\dav n/2} \prod_{i=1}^n \binom{d_i}{d_i/2} \biggr),
\]
which completes the proof.
\end{proof}

\subsection{Expansion properties of random graphs with given degrees}\label{S:same-order}

The next theorem establishes a lower bound on the isoperimetric constant of a random graph with given degrees.
Previous results have been restricted to random regular graphs, see \cite{bollo, KSVW, KW}.

\begin{thm}\label{L:cheeger}
Let $\dvec=\dvec(n)$ be a degree sequence and $C=C(n)>0$ be such that
we have $\dmin\ge 24+32C$ and $\dmax^2\le C\dav n$, where $\dav$ is
the average degree. Define $\alpha:=\frac{1}{6+8C}$.
If $G\sim\calG(n,\dvec)$, then
\[
  \P{ h(G) \ge \alpha\dmin} \ge 1 -  
   \Bigl(\Dfrac{\dmin^2}{5\dav n}\Bigr)^{\frac{5}{26}\alpha\dmin^2}. 
\]
\end{thm}

\begin{proof}
Let $S\subseteq \{1,\ldots,n\}$, and let $S_k$ be the set of all graphs with degree sequence~$\dvec$
such that $k$ is the number of edges between~$S$ and $\bar S:=\{1,\ldots,n\}-S$.
Assume that $s:=|S|\le \frac12 n$ and $k\le 2\alpha\dmin s$.
We can also assume that $s\ge (1-\alpha)\dmin$, since otherwise
at least $\alpha\dmin s$ edges leave~$S$.

Consider $G_k\in S_k$.  We can create a graph in $S_{k+2}$ by removing two
edges $v_1 v_2\in G_k[S]$, an $w_1w_2\in G_k[\bar S]$
and replacing them by either $v_1w_1$ and $v_2w_2$, or $v_1w_2$ and $v_2w_1$.
One or both of these operations may be unavailable due to an existing edge.
If $d'$ is the average degree of a vertex in $S$, the number of ways to perform the
operation is at least
\begin{align}
 \dfrac12 ( d's - k)(\dav n-d's-k) - k\dmax^2 
 &\ge \dfrac12 (\dmin s-k)(\dav n - \dmin s -k) - Ck\dav n \notag \\
 &\ge \dfrac14(1-4\alpha+4\alpha^2-8\alpha C) \dmin\dav s n \notag \\
 &= \Dfrac{1+C}{(3+4C)^2}  \dmin\dav s n, \label{Cheeg2}
\end{align}
where the first inequality holds since the first expression is increasing in~$d'$ for $s\le \frac12n$ and
$\dmax^2\le C\dav n$, while  the second line holds since  $k\le 2\alpha\dmin s$
and $s\le\frac12 n$.

Given $G_{k+2}\in S_{k+2}$ we can recover a graph in $G_k$ by performing the same
operation in reverse, which is determined by the choice of two edges between
$S$ and $\bar S$.  This time we need an upper bound, namely
\begin{equation}\label{Cheeg3}
   \binom{k+2}{2}.
\end{equation}
Combining~\eqref{Cheeg2} and~\eqref{Cheeg3}, we find that
\[
    \Dfrac{|S_k|}{|S_{k+2}|} \le \Dfrac{(3+4C)^2(k+2)^2}{2(1+C)\dmin\dav s n}.
\]

Define $K=K(s):=\lfloor \frac12( \alpha\dmin s-1)\rfloor$, which
implies by AM/GM that 
\[
   \prod_{i=1}^{2K}\, (\lfloor\alpha\dmin s\rfloor+i) \le \bigl(\dfrac 32 \alpha\dmin s\bigr)^K 
   \quad \text{~for $\alpha\dmin s\ge 2$}.
\]
This enables us to bound the probability that $k\le \alpha \dmin s$:
\begin{align*}
    \Dfrac{\sum_{k\le \alpha \dmin s} |S_k|}{\sum_{j\ge 0} |S_j|}
    \le \max_{k\le \alpha \dmin s} \Dfrac{|S_k|}{|S_{k+2K}|}
    &\le \biggl(\Dfrac{(3+4C)^2}{2(1+C)\dmin\dav s n}\biggr)^{\!K\,} \prod_{i=1}^{2K}\, (\lfloor \alpha\dmin s\rfloor +i)^2 \\
    &\le \biggl( \Dfrac{9 \dmin s}{ 32(1+C)\,\dav n}\biggr)^{\!K}.
\end{align*} 
Therefore, the probability that there is any set $S$ with $|S|=s$ having fewer than $\alpha\dmin s$ edges
leaving is  at most 
\[
    P(s) := 
      \binom ns \biggl( \Dfrac{9\dmin s}{32(1+C)\dav n}\biggr)^{\!\alpha\dmin s/2-3/2}.
\]
Using $(s+1)^{s+1}\le 4 s^{s+1}$ we find that $P(s+1)/P(s)$ is increasing with~$s$
and is bounded by $\frac{81}{128}$ at~$s=\frac12 n$.
Recalling that we can assume $s\ge (1-\alpha)\dmin$, we conclude that
\[
   \P{h(G)<\alpha\dmin} < 3\,P((1-\alpha)\dmin).
\]
It remains to simplify this bound.  First we apply $\binom ns \le \bigl(\frac {en}{s}\bigr)^s$.
Then we note that $\alpha\le\frac16$ and $\alpha\dmin\ge 4$ together imply that
$\frac12 \alpha(1-\alpha)\dmin^2 - (1-\alpha)\dmin - \frac32 \ge \frac{5}{26}\alpha\dmin^2$.
Then, under the same conditions,
\[
    3 \,\biggl(\Dfrac{e}{1-\alpha}\biggr)^{(1-\alpha)\dmin}
      \biggl( \Dfrac{9(1-\alpha)}{32(1+C)}\biggr)^{\alpha(1-\alpha)\dmin^2/2-3/2}
    \le 5_{\vphantom{x}}^{-\frac{5}{26}\alpha\dmin^2},
\]
This completes the proof.
\end{proof}

\nicebreak
\subsection{Completing the proof of Theorem~\ref{T:main}}\label{S:large}

Combining Corollary \ref{cor:eo} and Theorem \ref{L:cheeger}, we immediately get the following.

\begin{thm}\label{T:dense}
Let $G\sim\calG(n,\dvec)$
such that  $\dmax \gg \log^{8} n$ and   $\dmin \ge \gamma \dmax$ for some fixed constant $\gamma > 0$. 
Then, with probability at least $1 - e^{-\Omega(\dmax^2)}$,
\[
        \eo(G)  =  \paul(G) +  O\left( \dfrac{\log^2 \dmax} {\dmax}\right).
\]
\end{thm}

Finally   we show how to derive Theorem \ref{T:main}
from  Theorem \ref{T:sparse}  and Theorem \ref{T:dense}.
If the degree sequence $\dvec$  satisfies (R1), 
that is, $\dmax^2=o(n)$, then
we have $\dmax^2 = o(\dav n)$,
since all the degrees are positive.
Therefore, by Theorem \ref{T:sparse},
for fixed $\eps>0$,
\[
    \P{\eo(G) \ge \paul(G) + \eps}
 \le \dfrac{1}{{e^{n\eps}-1}}
 n^{1/2}e^{O(\dmax^2)} = e^{ - O(n) }.
\]

Assume that $\dmax^2=\Omega(n)$ and the degree sequence $\dvec$ satisfies (R2).
Theorem \ref{T:main} then follows from Theorem~\ref{T:dense} by noting that
  $\dmax^2=\Omega(\dmax^2+n)$ and
$(\log \dmax)^2 / \dmax=o(1)$.


\section{Numerical estimation via Eulerian partitions}\label{s:numeric}


The number of Eulerian orientations of a finite graph
is a $\#P$-complete problem equivalent to finding the
permanent of a 0--1 matrix~\cite{schrijver1983bounds,MihailW1996}.
However, the order of the matrix equals the number of edges in the graph, and the notorious difficulty of estimating large sparse permanents means that above about 100 edges we found it difficult to obtain accurate values.

Instead, we employed a repeatedly discovered theorem~\cite[Eq.\ (11)]{lieb1972two},~\cite{LasVergnas1981}, and  also~\cite{{bollo1}}. This result, stated below as Theorem~\ref{EOGvalue}, allowed us to obtain accurate estimates sometimes into thousand of vertices.

Let $G$ be a graph with even degrees $d_1,\ldots,d_n$.
An \textit{Eulerian partition} is a partition of the edges into undirected closed trails, where a \textit{trail} is a walk that doesn't repeat edges. Let $\calP(G)$ denote the set of all Eulerian partitions, and note that
\[
   \card{\calP(G)} = \prod_{i=1}^n\, \frac{d_i!}{(d_i/2)!\,2^{d_i/2}},
\]
since each partition is uniquely described by a pairing of the edges at each vertex.
For an Eulerian partition $P\in\calP(G)$, let $\card P$ denote the number of closed trails it comprises.

\begin{thm}\label{EOGvalue}
For any graph $G$ with even degrees $d_1,\ldots,d_n$,
\[
   \eo(G) = \paul(G) + \Dfrac1n \log T(G),~~\text{where}~~
   T(G) := \frac{1}{ \abs{\calP(G)} }
     \sum_{P\in\calP(G)} 2^{\card P}.
\]
\end{thm}
\begin{proof}
Say that an Eulerian partition $P$ and an Eulerian orientation $O$ are \textit{associated} if each trail in~$P$ is a directed closed walk in~$O$. In Figure~\ref{fig:pairing} we give an example of an Eulerian orientation and an associated edge pairing.
\begin{figure}[ht]
 \[ \includegraphics[scale=0.55]{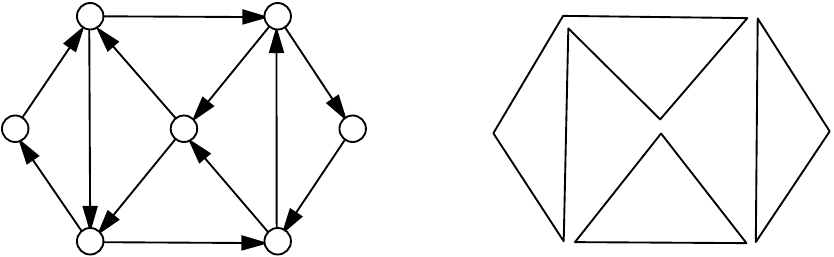} \]
 \caption{An Eulerian orientation and one of its associated Eulerian partitions.\label{fig:pairing}}
\end{figure}

Let $N$ be the number of pairs $(O,P)$, where $O$ is an Eulerian orientation and $P$ is an associated Eulerian partition.
Given $O$, all the associated Eulerian partitions are obtained from
a bijection at each vertex between the in-coming and out-going edges.  That is,
\[
   N = \EO(G)\; \prod_{i=1}^n\, (d_i/2)!\,.
\]
Conversely, given an Eulerian partition $P$, the number of associated Eulerian orientations is clearly $2^{\card P}$, so
\[
   N = \sum_{P\in\calP(G)} 2^{\card P}.
\]
In view of the definitions 
\eqref{def:eo} and \eqref{def:Pauling}, combining these counts gives the theorem.
\end{proof}

In Figure~\ref{fig:components} we show the distribution of $\card P$ for two graphs with 256 vertices.
$Q_8$ is the 8-dimensional hypercube (degree~8), while
$C_{16}\cart C_{16}$ is the two-dimensional square lattice
with periodic boundary conditions (degree~4).

\begin{figure}[ht]
\centering
\unitlength=1cm
  \begin{picture}(15,6.5)(0,0)
    \put(1,0.4){\includegraphics[scale=0.3]{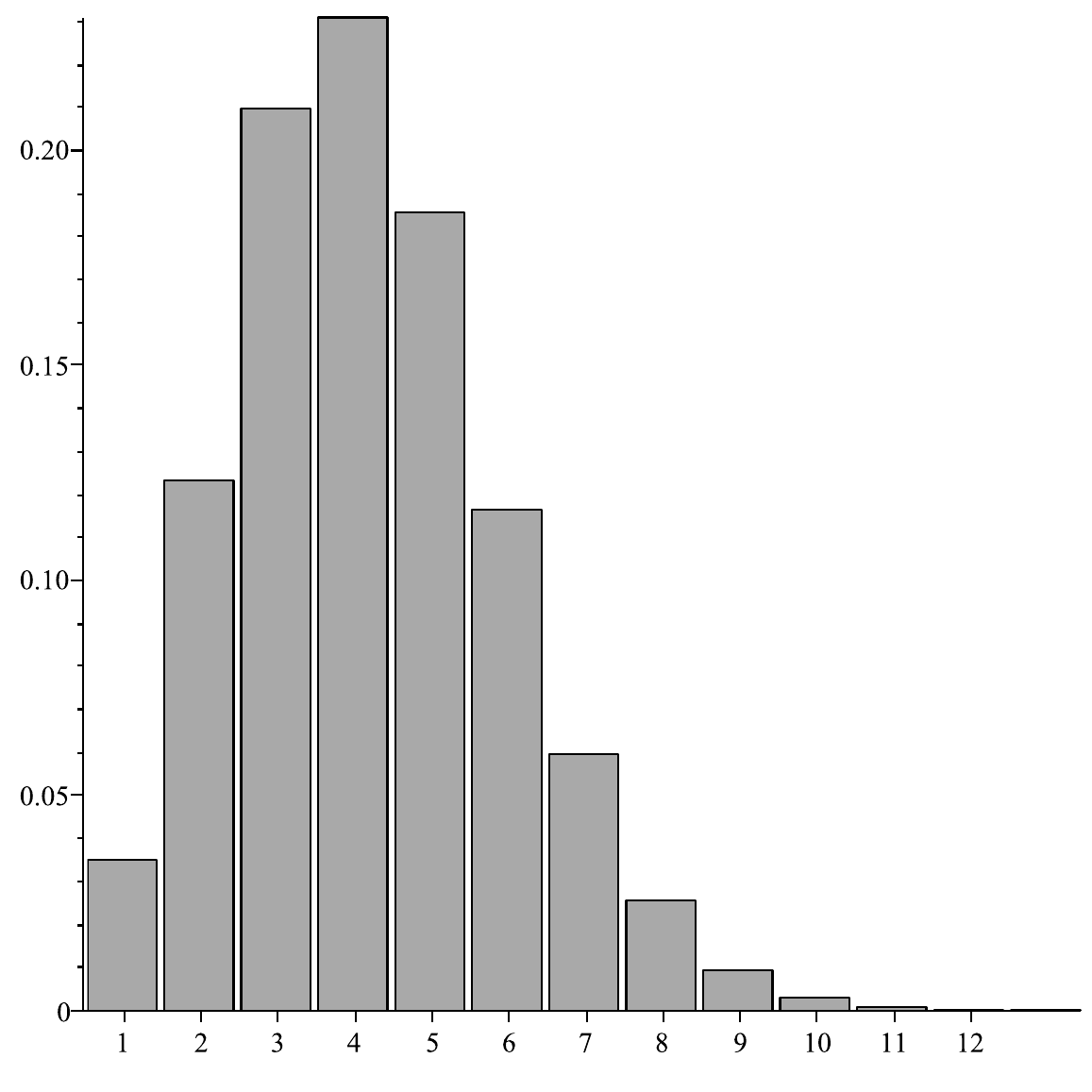}}
    \put(8,0.4){\includegraphics[scale=0.3]{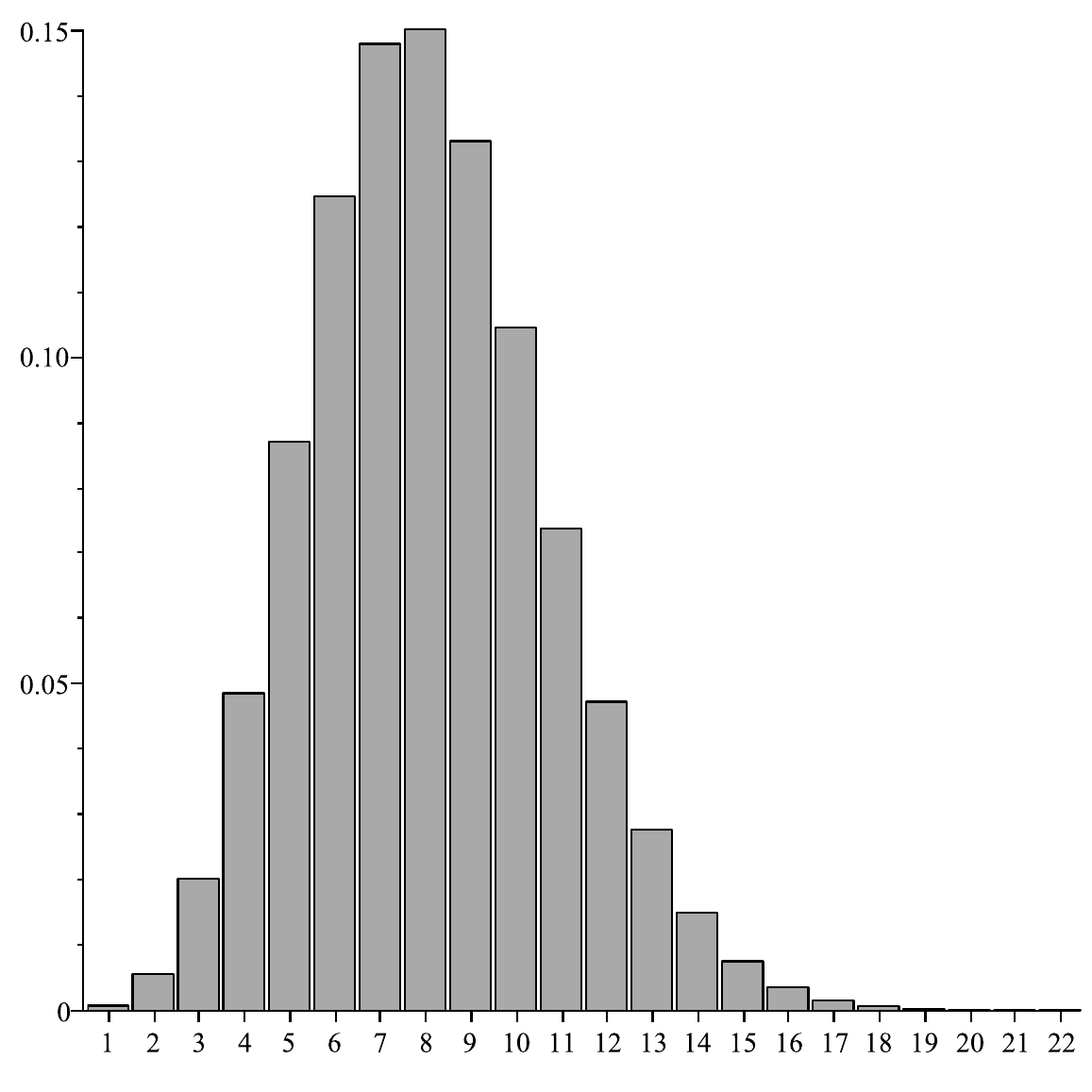}}
    \put(2.9,0.1){$Q_8$}
    \put(9.7,0.1){$C_{16}\cart C_{16}$}
  \end{picture}
  \caption{Distribution of $\card P$ for two 
  4-regular graphs on 256 vertices.\label{fig:components}}
\end{figure}

To apply Theorem~\ref{EOGvalue}, we can generate many edge partitions at random and calculate the average $2^{\card P}$.
With careful programming, each trial requires about $20\,\card{E(G)}$ nanoseconds.
Although $2^{\card P}$ is a highly skewed function, the average of a sequence of trials converges to a normal distribution as the number of trials increases.
In most cases, we found 1000 averages based on at least one million trials each, then from those 1000 averages we found the 2-sigma confidence interval for the mean.
The accuracy is better when $\card P$ is typically smaller, such as for higher degree or higher dimension.


\section{Products of cycles}\label{s:products}


In this section we consider simple graphs which
are Cartesian products of smaller graphs.
After some general theory we will focus on products of
cycles and their limits (infinite paths).
The most famous example is ``square ice'' (the two-dimensional square lattice)
which was the first periodic lattice whose residual entropy was determined exactly~\cite{lieb1967residual}.
In Section~\ref{s:cliques} we will consider a
different example that allows us to exemplify the
accuracy of our estimates when the degree increases.

Suppose $G$ and $H$ are two graphs.  The \textit{Cartesian product} $G\cart H$ has vertex set $V(G)\times V(H)$,
with $(u,v)$ adjacent to $(u',v')$ if either $u=u'$ and $v$ is adjacent to $v'$ in~$H$, or $v=v'$ and $u$ is adjacent to $u'$ in~$G$.

In order to compare the residual entropy with our
estimate $\ours$ defined in \eqref{ours}, we first consider the tree entropy.
Let $G$ be a simple graph with $n$ vertices. 
Recall the definition of the Laplacian matrix
$L(G)$ from Section~\ref{s:newbound}.

The next simple lemma gives 
the spanning tree count of graph products
in terms of eigenvalues of the Laplacian matrices.

\begin{lemma}\label{producttrees}
Let $G$ and $H$ be two connected graphs 
on $\ell$ and $m$ vertices, respectively.
Let $0=\mu_0, \mu_1, \ldots, \mu_{\ell-1}$
and $0=\mu'_0,\mu'_1, \ldots, \mu'_{m-1}$
be the eigenvalues of the Laplacian
matrices $L(G)$ and $L(H)$, respectively.
Then the number of spanning trees in the Cartesian product $G \cart H$~is
\[
t(G \cart H) = \Dfrac{1}{\ell m}
\prod_{\substack{0\le j <\ell,\, 0\le k <m \\[0.3ex] j+k\ne 0} }
 \!\!(\mu_j + \mu'_k).
\]
\end{lemma} 

\begin{proof}
By the Matrix Tree Theorem,
$\ell m\, t(G\cart H)$ is equal to the product of
the non-zero eigenvalues of $L(G\cart H)$.
In fact,
\[
   L(G\cart H) = I_\ell\otimes L(H) + L(G)\otimes I_m,
\]
where $\otimes$ is the matrix tensor product,
from which it follows that the
$\ell m$ eigenvalues of $L(G\cart H)$ are $\mu_j+\mu'_k$ 
for $0\le j<\ell$ and $0\le k<m$; see for example~\cite{Merris1998}.
\end{proof}

\subsection{Products of two cycles}\label{s:twocycles}

Let $C_m$ denote the cycle with $m$ vertices.
It is well known that the eigenvalues of $L(C_m)$ are
$2-2\cos\dfrac{2\pi j}{m}$ for $0\le j\le m-1$.

\begin{lemma}\label{tubetrees}
For $m\ge 3$,
\begin{align*}
    \lim_{\ell\to\infty} \tau(C_m\cart C_\ell)
  & = \Dfrac{1}{m}\sum_{j=1}^{m-1} \log \,g\Bigl(\Dfrac{2\pi j}{m}\Bigr), \text{~where} \\
  g(y) &= 2-\cos y + \sqrt{\cos^2 y-4\cos y+3}.
\end{align*}
Also, $\lim_{\ell,m\to \infty} \tau(C_m\cart C_\ell) \approx 1.1662436$.
\end{lemma}
\begin{proof}
By Lemma~\ref{producttrees}
\[
  \log t(C_m \cart C_\ell) = -\log(\ell m) +  
\sum_{\substack{0\le j <\ell,\, 0\le k <m \\[0.3ex] j+k\ne 0}}
\log\Bigl(
4 - 2 \cos \dfrac{2\pi j}\ell - 2\cos \dfrac{2\pi k}m \Bigr).
\]
As $\ell\to\infty$, the sum over $j$ can be replaced by an integral, using
\[
  \int_0^1 \log(1-a\cos(2\pi x))\,dx
   = \log\Dfrac{1+\sqrt{1-a^2}}{2}
 \quad(\abs a\le 1).
\]
This gives the first claim after some elementary manipulation.
The second claim comes from integrating $\log g(y)$ and
is exactly $\dfrac 4\pi$ times Catalan's constant~\cite{glasser2005entropy}.
\end{proof}

The value of $\lim_{m\to\infty} \eo(C_m\cart C_m)$ was famously determined by Lieb \cite{lieb1967residual} in 1967 to be exactly $\log\dfrac{8\sqrt 3}{9}\approx 0.43152$.
This compares poorly to Pauling's estimate $0.40547$
but very well to our estimate
$\lim_{m\to\infty} \ours(C_m\cart C_m)\approx 0.43054$.

\begin{table}[ht!]
\centering
\begin{tabular}{|c|c|c|c|}
\hline
$G$ & $\tau(G)$ & $\eo(G)$  & $\ours(G)$ \\ 
\hline\hline
$C_{3}\cart C_\infty$ &  1.04453 &  0.46210 &  0.49140 \\
$C_{4}\cart C_\infty$ &  1.09917 &  0.46299 &  0.46408 \\
$C_{5}\cart C_\infty$ &  1.12373 &  0.44216 &  0.45180 \\
$C_{6}\cart C_\infty$ &  1.13687 &  0.44577 &  0.44523 \\
$C_{7}\cart C_\infty$ &  1.14472 &  0.43690 &  0.44130 \\
$C_{8}\cart C_\infty$ &  1.14979 &  0.43960 &  0.43877 \\
$C_{9}\cart C_\infty$ &  1.15326 &  0.43477 &  0.43703 \\
$C_{10}\cart C_\infty$ &  1.15574 &  0.43672 &  0.43579 \\
$C_{11}\cart C_\infty$ &  1.15757 &  0.43369 &  0.43488 \\
$C_{12}\cart C_\infty$ &  1.15895 &  0.43514 &  0.43419 \\
$C_{13}\cart C_\infty$ &  1.16003 &  0.43308 &  0.43365 \\
$C_{14}\cart C_\infty$ &  1.16089 &  0.43418 &  0.43322 \\
$C_{15}\cart C_\infty$ &  1.16158 &  0.43269 &  0.43287 \\
$C_{16}\cart C_\infty$ &  1.16215 &  0.43356 &  0.43259 \\
\hline
\end{tabular}
\caption{Parameters for $C_m\cart C_\ell$ as $\ell\to\infty$
\label{table:tubes}}
\end{table}

\begin{figure}[ht!]
\centering
\unitlength=1cm
\begin{picture}(13,8)(0,0)
\put(0,0){\includegraphics[scale=0.63]{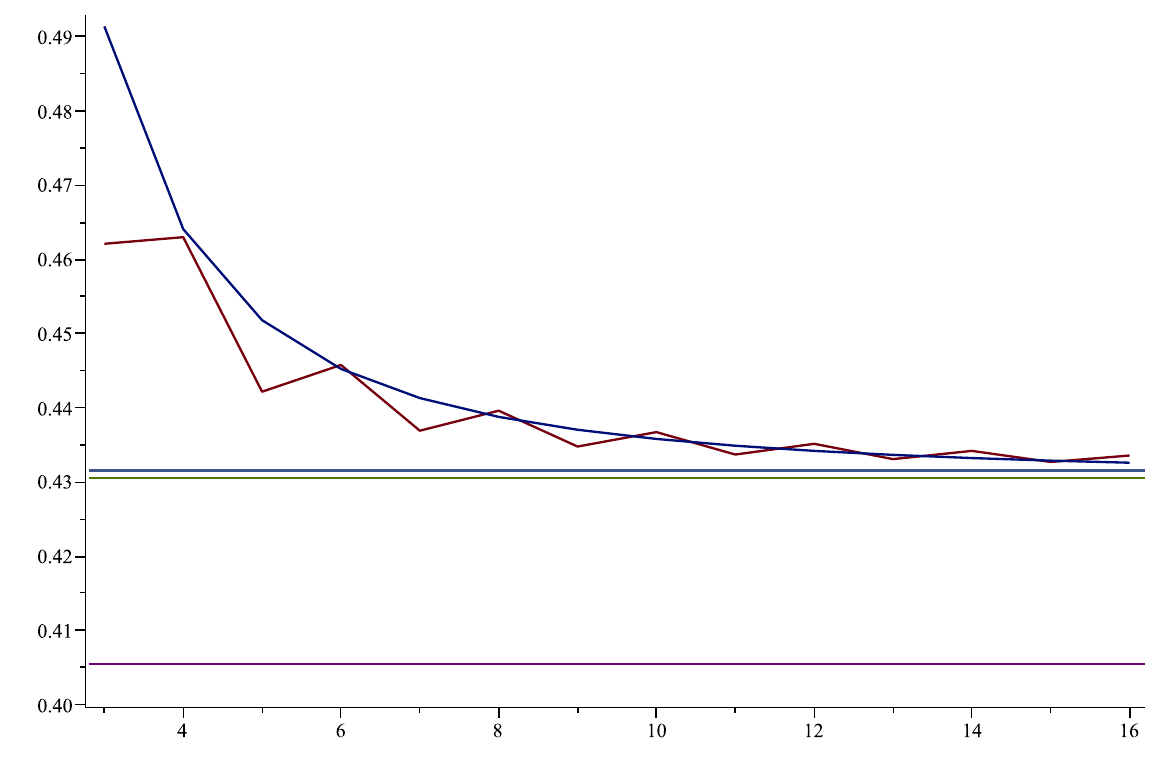}}
\put(1.2,4.3){\small exact $\eo$}
\put(1.9,6.5){\small $\ours$ estimate}
\put(6,-0.1){$m$}
\put(1.2,2.7){\small exact limit (Lieb)}
\put(1.2,3.35){\small limit of $\ours$}
\put(1.2,1.3){\small Pauling}
\end{picture}
\caption{Exact and estimated residual
 entropies for tubes $C_m\cart C_\infty$
    \label{fig:enter-label}}
\end{figure}
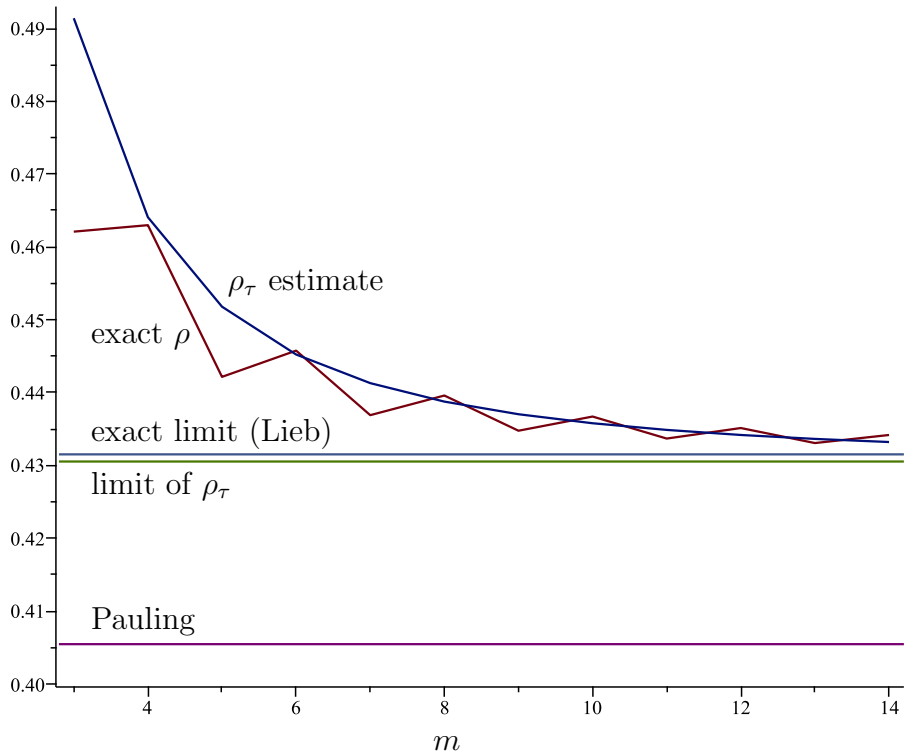

To further illustrate the usefulness of our estimate,
we considered the case where the square lattice is finite in one direction; i.e., a square lattice on an infinitely long cylinder.
For $3\le m\le 14$, we obtained precise values of 
$\lim_{\ell\to\infty} \eo(C_m\cart C_\ell)$ using the transfer matrix method
which we describe in the following theorem.
These values and the corresponding estimates $\ours$ are presented in Table~\ref{table:tubes} and Figure~\ref{fig:enter-label}.
It is seen that $\ours$ tracks $\eo$ very well as $m$ increases, especially for large~$m$.

\begin{thm}\label{thm:transfer}
    Let $G$ be an Eulerian graph with vertices $\{1,\ldots,n\}$.
    For $\zvec=(z_1,\ldots,z_n)\in\Integers^n$,
    define $N_G(\zvec)$ to be the number of
    orientations of $G$ such that $\dout(v)-\din(v)=z_v$
    for $1\le v\le n$, where $\din(v),\dout(v)$ are the in-degree and
    out-degree of vertex~$v$.
    Define the $2^n\times 2^n$ matrix $T=(t_{\xvec,\yvec})$, whose rows
    and columns are indexed by $\xvec,\yvec\in\{0,1\}^n$, by
    $t_{\xvec,\yvec}=N_G(2(\yvec-\xvec))$.
    Then
    \[
       \lim_{\ell\to\infty} \eo(G\cart C_\ell) = \dfrac1n\log\lambda,
    \]
    where $\lambda$ is the largest eigenvalue of~$T$.
    Furthermore, let $\varGamma$ be the group action on $\{0,1\}^n$
    induced by the automorphism group of $G$ acting on the coordinate
    positions, together with the
    involution $\xvec\mapsto(1,\ldots,1)-\xvec$.
    Suppose this action has orbits $O_1,\ldots,O_m$.
    Define the $m\times m$ matrix $S=(s_{i,j})$ where $s_{i,j}$ is the
    common row sum of the submatrix of $T$ induced by rows~$O_i$ and
    columns~$O_j$.
    Then $S$ has the same largest eigenvalue~$\lambda$.
\end{thm}
\begin{proof}
  The rationale for $T$ was described by Lieb~\cite{lieb1967residual}
  and we will be content with sketching a proof of the last part.
  Note that, by symmetry and converse,
  $t_{\xvec^\gamma,\yvec^\gamma}=t_{\xvec,\yvec}$ for $\xvec,\yvec\in\{0,1\}^n$
  and $\gamma\in\varGamma$.
  This implies that if $\vvec$ is a positive eigenvector of $T$ with eigenvalue
  $\lambda$, then so is~$\vvec^\gamma$.
  By averaging over $\gamma\in\varGamma$, we find a non-zero eigenvector of~$T$
  corresponding to eigenvalue $\lambda$ which takes a constant value
  $r_i$ on each orbit~$O_i$.
  Then $(r_1,\ldots,r_n)$ is an eigenvector of~$S$ with eigenvalue~$\lambda$.
  Conversely, any eigenvector of $S$ becomes an eigenvector of $T$ with the
  same eigenvalue on replicating its value in each orbit.
\end{proof}
The advantage of the matrix $S$ is that it can be much smaller than $T$.
For example, the final row in Table~\ref{tab:3cycles} below has $T$ matrix
of order 1,048,576 but its $S$ matrix has order 7,456.

\nicebreak
\subsection{Products of three cycles}
 
Define $R_m=C_m\cart C_m\cart C_m$ to be a finite simple
cubic lattice with periodic boundary conditions.
We did not find an estimate of the residual entropy of this
lattice in the literature.
From \eqref{eq:Sch} and \eqref{eq:LVupper} we have
\[
     0.91623 \le \lim_{m\to\infty} \eo(R_m) \le 1.0925.
\]
From \cite{rosengren1987number} we have
$\lim_{m\to\infty}\tau(R_{m})\approx 1.67338$,
so our estimate is $\lim_{m\to\infty}\ours(R_m)=0.9251$.

\begin{figure}[ht!]
\centering
\unitlength=1cm
\begin{picture}(12.5,8.5)(0,0)
\put(0,0){\includegraphics[scale=0.60]{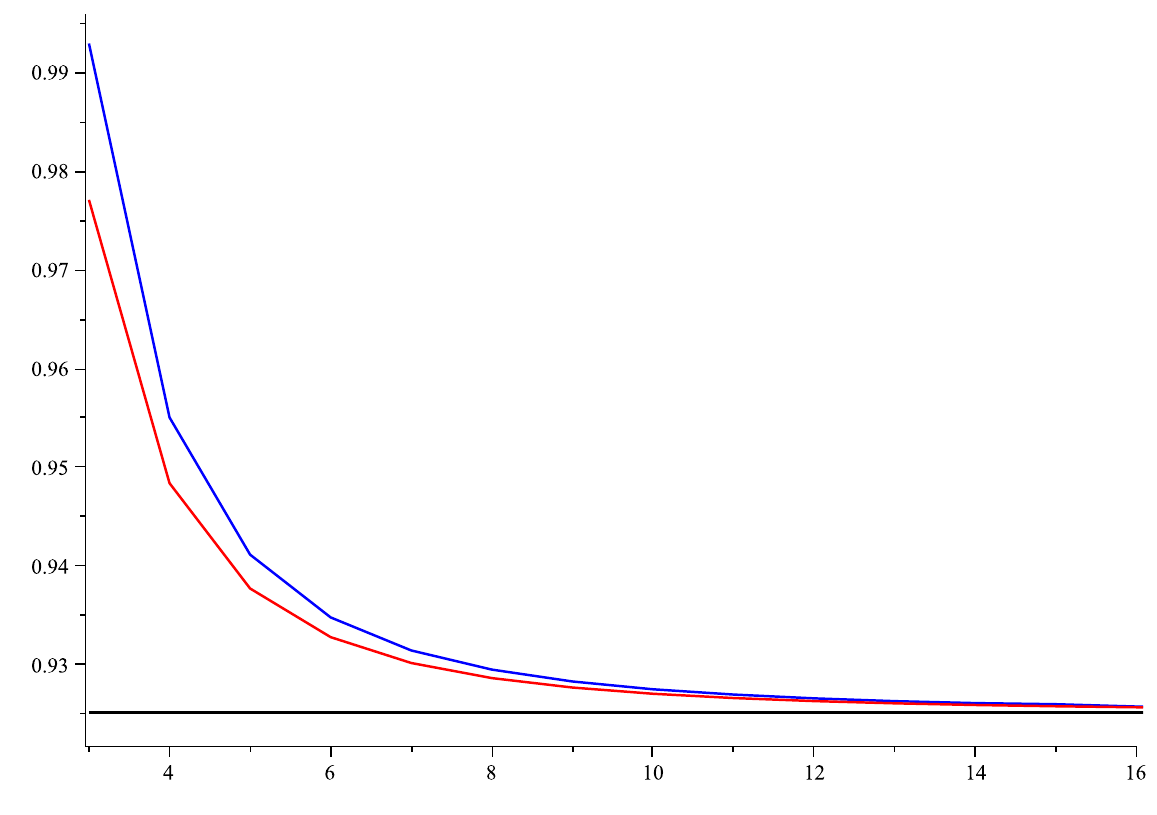}}
\put(1.7,5.0){\small $\eo(R_m)$}
\put(1.1,2.1){\small $\ours(R_m)$}
\put(6,-0.1){$m$}
\put(1.2,1.2){\small limit of $\ours$}
\end{picture}
\caption{Exact and estimated residual
 entropies for the simple cubic lattice
 $R_m=C_m\cart C_m\cart C_m$.\label{fig:simplecubic}}
\end{figure}
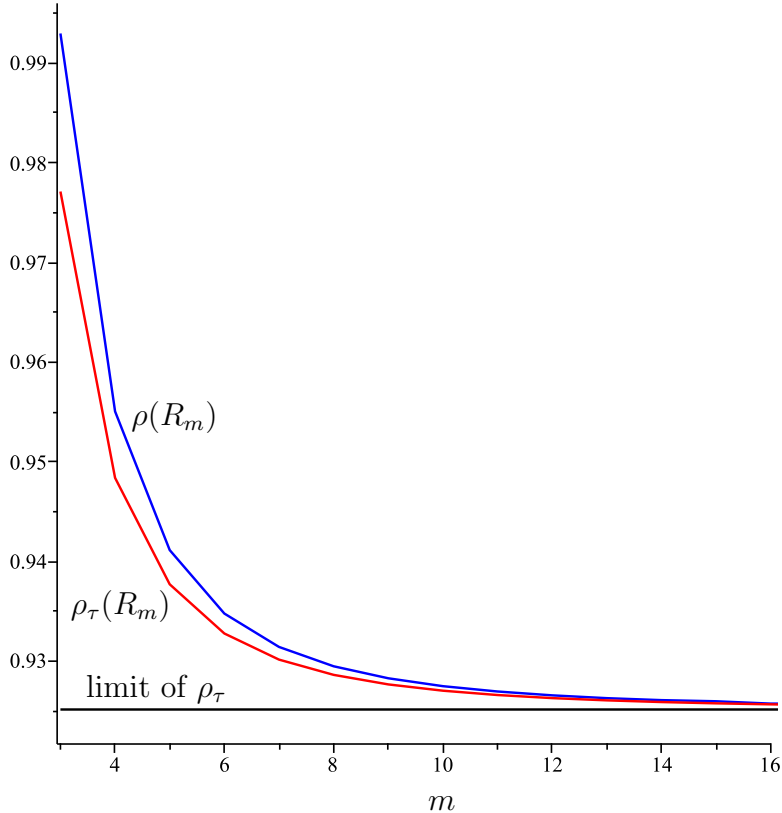

To judge the accuracy of our estimate, we computed values
of $\eo(R_m)$ up to $m=20$ using the method of Section~\ref{s:numeric}.
Precise values become difficult to obtain past $n=16$ (4096 vertices).
In Figure~\ref{fig:simplecubic} we compare $\eo(R_m)$ to $\ours(R_m)$ and again observe good correlation
between $\ours$ and~$\eo$.
Using rational extrapolation of $\eo(R_m)$, we believe that
\[
    \lim_{m\to\infty} \eo(R_m) = 0.9252\pm 0.0002.
\]
In other words, we cannot distinguish between $\lim_{m\to\infty} \eo(R_m)$
and $\lim_{m\to\infty} \ours(R_m)$.

 \begin{table}[ht!]
       \centering
     \medskip
     \begin{tabular}{ |c|c|c|c|c| }
     \hline
   $G$ & $\tau(G)$ & $\eo(G)$ & $\ours(G)$
         \\ \hline\hline
    $C_3 \cart C_3 \cart C_\infty$   
         & 1.61344 & 0.95055 & 0.95511 \\
    $C_3 \cart C_4 \cart C_\infty$  
         & 1.63332 & 0.94486 & 0.94521  \\
    $C_3 \cart C_5 \cart C_\infty$ 
         & 1.64164 & 0.93930 & 0.94101 \\
    $C_3 \cart C_6 \cart C_\infty$ 
         & 1.64605 & 0.93857 & 0.93881 \\
    $C_4 \cart C_4 \cart C_\infty$ 
         & 1.64941 & 0.93703 & 0.93713 \\
    $C_4 \cart C_5 \cart C_\infty$ 
         & 1.65593 & 0.93382 & 0.93387 \\
         \hline
     \end{tabular}
     \caption{Parameters for $\lim_{n\to\infty} C_m\cart C_\ell\cart C_n$\label{tab:3cycles}}
 \end{table}

\medskip
We also show in Table~\ref{tab:3cycles} results for the Cartesian product of three cycles where two of the cycles have small fixed length.
We computed $\tau(G)$ using a simple extension of Lemma~\ref{tubetrees}, and $\eo(G)$ using the transfer matrix method.
Note that $\paul(G)=0.91629$ in all cases.
 

\nicebreak
\section{Cycles of cliques}\label{s:cliques}

In this section, we find the residual entropy of cycles of cliques which are products $K_m \cart C_\ell$.
Our reason for studying this family is that increasing~$m$ allows us to test Conjecture~\ref{Conj} and to observe how $\ours$ and $\eo$ are related as the degree increases.
Throughout this section, we assume that $m$ is odd since the number of Eulerian orientations is zero otherwise.  
\begin{figure}[ht!]
    \[
    \includegraphics[scale=0.35]{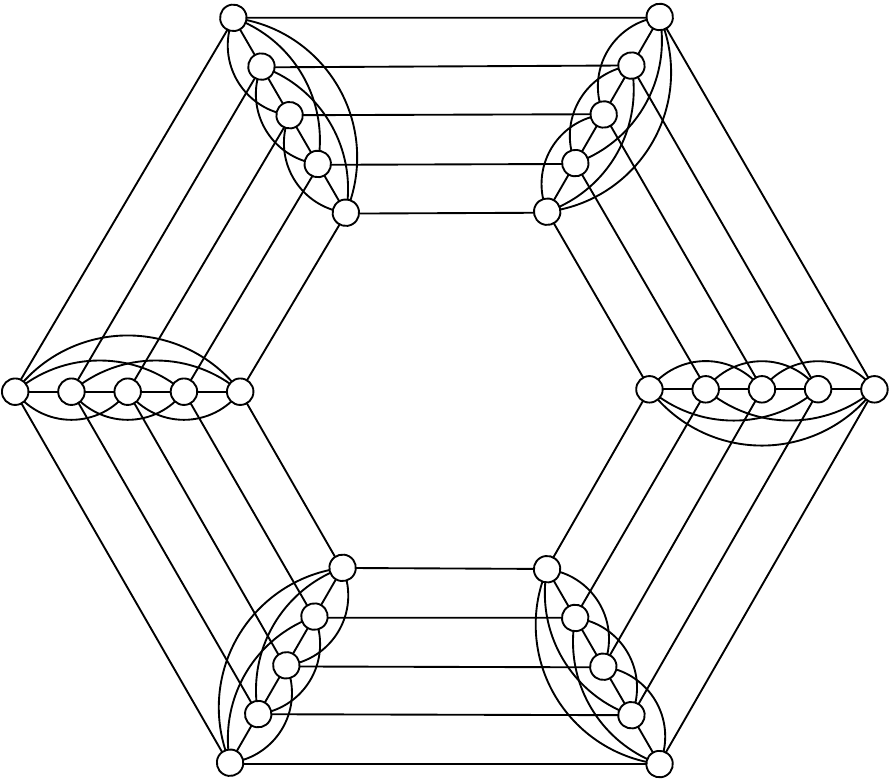}
  \]
   \caption{The Cartesian product $K_5\cart C_6$}
\end{figure}

\subsection{Residual entropy}

Recall that $\RT(m)= \EO(K_m)$ is the number of regular tournaments.

\begin{thm}\label{T:cycle-clique}
 If  $n = m \ell \to \infty$ and $m$ is odd then 
\[ 
    \eo(K_m \cart C_{\ell}) =    \Dfrac{1}{m}\log \biggl( \RT(m+2) \Bigm/\binom{m+1}{(m{+}1)/2}\biggr) + O\Bigl(\Dfrac{\log m}{m\ell}\Bigr).
\]
\end{thm}

We start the proof with a sequence of auxiliary lemmas.

\begin{lemma}\label{L:tournaments}
Let $\fvec \in \{-2,0,2\}^m$ be such that $\sum_{i \in [m]} f_i = 0$.
Let $\NT(m,\fvec)$ denote the number of tournaments with $m$ vertices that $\fvec$ is the vector of  differences of out-degrees and in-degrees.  
Then,
\[
    \NT(m,\fvec)= 
    \Dfrac{2}{\binom{m+1}{(m+1)/2}^2}\,
    e^{c(\fvec)} \RT(m+2),
\]
where $|c(\fvec)| \le 2 + O(m^{-1})$ uniformly over such $\fvec$ as $m \to \infty$.
\end{lemma}
\begin{proof}
  From \cite[Theorem 4.4]{McKayW1996}, we find that 
  \begin{equation}\label{NTmf}
    \NT(m,\fvec) = \RT(m)\exp\biggl(-\Dfrac{\sum_{i \in [m]}f_i^2}{2m} + O(m^{-1})\biggr).
  \end{equation}
  Note that $\dfrac{\sum_{i \in [m]}f_i^2}{2m}  \in [0,2]$.

  Next, consider a tournament $T$ with vertices $V\cup\{u,v\}$,
  where $\card V=m$. Let $T'$ be the subtournament of $T$ induced by~$V$.
  There are $\binom{m+1}{(m+1)/2}^2/2$ ways to choose
  the neighbours of $u$ and $v$ so that their in-degrees
  and out-degrees are equal.
  For each of those cases, a particular
  $\fvec(T')\in\{-2,0,2\}^m$ is necessary and sufficient
  for~$T$ to be regular.
  That is, $\RT(m+2)$ is the sum of 
  $\binom{m+1}{(m+1)/2}^2/2$ terms, each having
  value given by~\eqref{NTmf}. 
  This completes the proof.
\end{proof}

Now we consider a cycle of cliques $K_m\cart C_\ell$.
Take a clockwise cyclic
ordering of the cliques $K_{m}^1, \ldots, K_{m}^{\ell}$.
Given any orientation $D$ of the cycle of cliques, we define the \textit{net flow} $f_{(i-1) \to i}(D)$
as the difference of the number of clockwise arcs 
and anticlockwise arcs
between two consecutive cliques $K_{m}^{i-1}$ and $K_{m}^{i}$
in $D$ for all $i-1, i$ modulo~$\ell$.

\begin{lemma}
\label{lm:flow}
If $D$ is an Eulerian orientation of $K_m \cart C_\ell$ then,
for all $i, j\in [\ell]$,
\[ 
f_{(i-1) \to i}(D) = f_{(j-1) \to j}(D).
\] 
\end{lemma}

\begin{proof}
It suffices to show that 
$f_{(i-1)\to i}(D)
= f_{i \to (i+1)}(D)$ for all $i \in [\ell]$.
Using that $D$ is Eulerian, we observe that
\[
f_{i \to (i+1)}(D) - f_{(i-1) \to i}(D)
= 
\sum_{v \in K_{m}^i}  d^+(v) 
- \sum_{v \in K_{m}^i} d^-(v) = 0,
\]
where $d^+$ and $d^-$ denote the out-degree and in-degree of a vertex 
in $D$, respectively.
\end{proof}

We introduce a class $A_{m,f}$ of digraphs,
for $m,f$ odd and $|f|\le m$. These digraphs correspond to
an Eulerian orientation of $K_m\cart C_\ell$ induced on the clique $K_m^{i}$ and with previous and subsequent cliques replaced by two vertices $s$ and $t$.
Such graphs have $m+2$ vertices including these two special vertices. There is no directed edge between $s$ and $t$, but exactly one directed edge between each other pair. 
Vertices $v\ne s,t$ have $d^-(v)=d^+(v)$, while
$d^+(s)=d^-(t)=\dfrac{m+f}{2}$ and
$d^-(s)=d^+(t)=\dfrac{m-f}{2}$.
Next, we study the rates of decrease of the quantity
$a_{m,f}:=\card{A_{m,f}}$
decreases with respect to $f$, using  the following lemma. 
\begin{lemma}\label{switchlemma}
  Suppose a bipartite graph $G$ with two parts $X,Y$ has no isolated vertices.
  Suppose that there is a constant $c$ such that $d(x)\le cd(y)$ for every edge $xy$ with $x\in X,y\in Y$.
  Then $|Y|\le c\,|X|$.
\end{lemma}
\begin{proof}
We have
\[
 |Y| = \sum_{xy\in G} \Dfrac{1}{d(y)}
    \le c\sum_{xy\in G} \Dfrac{1}{d(x)} = c\,|X|
\]
as required.
\end{proof}

\begin{lemma}\label{regisbest}
For odd $m,f$ with $|f|\le m$,\\
(a) $a_{m,-f} = a_{m,f}$.\\[0.6ex]
(b) If $f\ge 3$ then
\[  \Dfrac{a_{m,f}}{\binom{m}{(m+f)/2}} \le     \Dfrac{a_{m,f-2}}{\binom{m}{(m+f-2)/2}}
\le \Dfrac{a_{m,1}}{\binom{m}{(m+1)/2}}
= \Dfrac{\RT(m+2)}{ \binom{m+1}{(m+1)/2}}.
 \]
\end{lemma}

\begin{proof}
Part (a) is proved by interchanging the roles of $s$ and~$t$.

We next prove the first inequality of part (b).
Define a bipartite graph $G$ with parts $A_{m,f}$ and $A_{m,f-2}$.
For $x\in A_{m,f}$ and $y\in A_{m,f-2}$, $x$ is adjacent to $y$ if $y$ is obtained by reversing a directed path of length~2 from~$s$ to~$t$.
Consider such an edge~$xy$.
The vertex set $V(x)\setminus \{s,t\}$ can be partitioned into four parts:
$V_1$ is adjacent from $s$ and to $t$; $V_2$, with $\dfrac{m+f}{2}-\card{V_1}$ vertices, is adjacent from both $s$ and~$t$; $V_3$, with $\dfrac{m+f}{2}-\card{V_1}$ vertices, is adjacent to both $s$ and~$t$; and finally $V_4$, with $\card{V_1}-f$ vertices, is adjacent to $s$ and from~$t$.
Note that $\card{V_1}=\card{V_4}+f>0$ so $G$ has no isolated vertices in the first part, and the same argument with the parts interchanged shows that there are no isolated vertices in the second part either.

Directed paths of length 2 from $s$ to $t$ correspond to vertices in $V_1$.
Reversing one path takes us to~$y$, where there are $\card{V_4}+1$ directed paths from $t$ to $s$.
Thus, $d(x)=\card{V_1}$ and $d(y)=\card{V_4}+1=\card{V_1}-f+1$.

Lemma~\ref{switchlemma} now gives us
\[
   \Dfrac{a_{m,f}}{a_{m,f-2}}
   \le
   \max_{A_{m,f}} \Dfrac{\card{V_1}-f+1}{\card{V_1}}.
\]
This expression is increasing with $\card{V_1}$,
which is at most $\dfrac{m+f}{2}$.
Observe that the ratio of the corresponding binomials is also $\frac{m-f+2}{m+f}$. This  establishes the claimed monotonicity of 
$ \frac{a_{m,f}}{\binom{m}{(m+f)/2}}$  with respect to $f$. 

To establish claim we observe that any digraph from $A_{m,1}$ corresponds to a regular tournament on $m+2$ vertices with a specified direction between $s$ and $t$. Thus,
\[ 
    \Dfrac{a_{m,1}}{\binom{m}{(m+1)/2}} = 
    \Dfrac{\RT(m+2)}{2\binom{m}{(m+1)/2}} =    \Dfrac{\RT(m+2)}{ \binom{m+1}{(m+1)/2}},
\]
as required.
\end{proof}

Now we are ready to  establish the residual entropy of a cycle of cliques.
\begin{proof}[Proof of Theorem \ref{T:cycle-clique}]
Let $\EO_f(m,\ell)$ denote the number of Eulerian orientations of $K_m\cart C_{\ell}$ with net flow $f \in [-m,m]$. 
Note that from Lemma \ref{lm:flow} we know that the flow between any two consecutive cliques is the same.
Therefore,
\begin{equation}\label{flow-sum}
    \EO(K_m\cart C_\ell) = \sum_{\substack{|f|\le m\\\text{$f$ is odd}} } \EO_f(m,\ell).
\end{equation}
Any orientation with flow $f$
can be represented as a sequence of $\ell-1$ choices of    $D^{i}\in A_{m,f}$, $i =1,\ldots,\ell-1$,
and then a choice of orientations of the edges of the last clique $K_m^{\ell}$ with specified differences of out and in degrees from $\{-2,0,2\}$. Note that  
the orientations of edges of $t$-vertices of $D^i \in A_{m,f}$   should match the orientations
of the $s$-vertices of $D^{i+1}$ so we need to adjust the count
by a factor $\binom{m}{(m-f)/2}$ for the choices of $D^{i+1}$ given $D^i$.
Therefore, 
\[
    \EO_f(m,\ell) = \biggl(\Dfrac{a_{m,f}^{\ell-1}}{\binom{m}{(m{-}f)/2}}\biggr)^{\ell-1} \binom{m}{(m{-}f)/2} N_f(m,\ell),
\]
where $N_f(m,\ell)$ is an average number of choices for 
the orientations of the edges  of  the last clique $K_m^{\ell}$ given the orientations of all other edges of $K_m\cart C_{\ell}$. 
Using Lemma \ref{L:tournaments} and Lemma~\ref{regisbest}, we obtain
\[
    \EO_f(m,\ell)  \le \EO_1(m,\ell) e^{2 + O(m^{-1})} 
      =  
    \biggl(\Dfrac{\RT(m+2)}{ \binom{m+1}{(m+1)/2}}\biggr)^{\ell} e^{O(1)}.
\]
Substitution this bound into  \eqref{flow-sum}, we find that
\[
   \biggl(\Dfrac{\RT(m+2)}{ \binom{m+1}{(m+1)/2}}\biggr)^{\ell} e^{O(1)}
    \le  \EO(K_m\cart C_\ell) \le m \biggl(\Dfrac{\RT(m+2)}{ \binom{m+1}{(m+1)/2}}\biggr)^{\ell} e^{O(1)}.
\]
Taking the logarithm and dividing by the number of vertices $n = m \ell$ completes the proof.
\end{proof}


\subsection{Pauling's estimate and the tree-entropy correction}
A cycle of cliques $K_m\cart C_\ell$ has degree $m+1$. Therefore, Pauling's estimate \eqref{def:Pauling} is 
\[
\paul(K_m\cart C_\ell) = \log \binom{m+1}{(m+1)/2} - \Dfrac{m+1}{2} \log 2. 
\]
In Table~\ref{tab:KmCell} we compare the exact residual entropy given by Theorem \ref{T:cycle-clique}  with Pauling's estimate \eqref{def:Pauling} and our estimate \eqref{ours} for $m$ up to $35$.
 The exact values of $\RT(m+2)$ are taken from \cite[Table 1]{IMZ}. 
Recall  that our estimate $\ours$ requires the spanning tree entropy, which is given by the next lemma.


\begin{table}[ht!]
 \centering
 \begin{tabular}{|c|c|c|c|c|c|}
 \hline
   \vrule width0pt height2ex depth1ex 
   $G$ & degree & $\tau(G)$ & $\eo(G)$ & $\paul(G)$ & $\ours(G)$ \\
   \hline \hline
$K_{3}\cart C_\infty$ & 4 & 1.04453 & 0.46210 & 0.40547 & 0.49140 \\
$K_{5}\cart C_\infty$ & 6 & 1.53988 & 0.97656 & 0.91629 & 0.99189 \\
$K_{7}\cart C_\infty$ & 8 & 1.87255 & 1.53422 & 1.47591 & 1.54351 \\
$K_{9}\cart C_\infty$ & 10 & 2.12402 & 2.11892 & 2.06369 & 2.12514 \\
$K_{11}\cart C_\infty$ & 12 & 2.32634 & 2.72190 & 2.66983 & 2.72635 \\
$K_{13}\cart C_\infty$ & 14 & 2.49561 & 3.33800 & 3.28887 & 3.34134 \\
$K_{15}\cart C_\infty$ & 16 & 2.64109 & 3.96395 & 3.91748 & 3.96655 \\
$K_{17}\cart C_\infty$ & 18 & 2.76862 & 4.59755 & 4.55347 & 4.59963 \\
$K_{19}\cart C_\infty$ & 20 & 2.88213 & 5.23725 & 5.19532 & 5.23896 \\
$K_{21}\cart C_\infty$ & 22 & 2.98438 & 5.88195 & 5.84195 & 5.88337 \\
$K_{23}\cart C_\infty$ & 24 & 3.07739 & 6.53079 & 6.49253 & 6.53199 \\
$K_{25}\cart C_\infty$ & 26 & 3.16268 & 7.18313 & 7.14646 & 7.18416 \\
$K_{27}\cart C_\infty$ & 28 & 3.24143 & 7.83847 & 7.80324 & 7.83936 \\
$K_{29}\cart C_\infty$ & 30 & 3.31457 & 8.49640 & 8.46249 & 8.49718 \\
$K_{31}\cart C_\infty$ & 32 & 3.38283 & 9.15658 & 9.12388 & 9.15727 \\
$K_{33}\cart C_\infty$ & 34 & 3.44682 & 9.81876 & 9.78718 & 9.81937 \\
$K_{35}\cart C_\infty$ & 36 & 3.50704 & 10.48270 & 10.45215 & 10.48325 \\
\hline
 \end{tabular}
 \caption{Residual entropy $\lim_{\ell\to\infty}\eo(K_m\cart C_\ell)$ compared to its estimates
 $\paul$ and $\ours$.\label{tab:KmCell}}
\end{table}
 
\begin{lemma}\label{L:tree-cycle-clique}
For $\ell,m\ge 3$, the number of spanning trees in $K_m\cart C_\ell$ is
\[
    t(K_m\cart C_\ell) = \Dfrac{\ell}{m} \Bigl( \Bigl(\Dfrac{\sqrt m+\sqrt{m+4}}{2}\Bigl)^{\!2\ell}
 +   \Bigl(\Dfrac{\sqrt m+\sqrt{m+4}}{2}\Bigl)^{\!-2\ell\,}\Bigr)^{m-1}.
\]
Moreover,
\[
   \tau(K_m\cart C_\ell)
     = \begin{cases}
     \dfrac{\log(\ell/m)}{\ell m} + \dfrac{2(m-1)}{m}\log \dfrac{\sqrt m+\sqrt{m+4}}{2} + O(\ell^{-1}m^{-2\ell}),
         &  \text{as \,$\ell m\to\infty$;} \\[1ex]
       \dfrac{\log(\ell/m)}{\ell m}+\log m - \dfrac{\log m-2}{m} + O(m^{-2}),
          &\text{as \,$m\to\infty$.}
       \end{cases}
\]
\end{lemma}
\begin{proof}
  The non-zero eigenvalues of $L(C_\ell)$ are
  $2-2\cos\bigl(\frac{2\pi j}{\ell}\bigr)$ for $1\le j\le\ell-1$, while the non-zero eigenvalues of $L(K_m)$ are all equal to~$m$.
  Therefore, we have from Lemma~\ref{producttrees} that
  \begin{align*}
    t(K_m\cart C_\ell) &= \ell m^{m-2} \prod_{j=1}^{\ell-1}\, \Bigl(m+2-2\cos\dfrac{2\pi j}{\ell}\Bigr)^{m-1} \\
       &= \ell m^{-1} T(m,\ell)^{m-1},
   \end{align*}
   where
   \[
        T(m,\ell) = \prod_{j=0}^{\ell-1} \, \Bigl(m+2-2\cos\dfrac{2\pi j}{\ell}\Bigr).
  \]
  The latter product can be recognised as $2(-1)^\ell T_{2\ell} (i m^{1/2}/2) - 2$, where
  $T_{2\ell}(x)$ is the Chebyshev polynomial in its standard normalization. 
  The lemma now follows from the explicit form of $T_{2\ell}$.
\end{proof}

Finally, for large  $m$, we show  that Pauling's estimate  \eqref{def:Pauling} approximates the exact residual entropy up to an error $O\bigl(\frac{\log m}{m}\bigr)$, thus confirming Conjecture \ref{Conj} for cycles of cliques.
Our new heuristic $\ours$ has a much smaller error,
again demonstrating its remarkable precision.

\begin{thm}\label{T:estimates:cycle-clique}
If $m$ is odd and $m\to \infty$ then
\begin{align*}
\eo(K_m \cart C_{\ell}) &= \dfrac{m+2}{2} \log 2  
- \dfrac{m-1}{2m} \log m   - \dfrac{\log \pi}{2}- \dfrac{3}{2m}
      + O\left(\dfrac{1}{m^2} + \dfrac{\log m}{m \ell}\right) 
\\
&=\paul(K_m \cart C_{\ell}) + \dfrac{\log m  }{2m}   -\dfrac{3}{4m} + O\left(\dfrac{1}{m^2} + \dfrac{\log m}{m \ell}\right) 
\\&= \ours(K_m \cart C_\ell) + O\left(\dfrac{1}{m^2} + \dfrac{\log m}{m \ell}\right),
\end{align*}
where $\paul$  and $\ours$ are defined in \eqref{def:Pauling}  and \eqref{ours}.
\end{thm}
\begin{proof}
The value of $\RT(m+2)$ follows from~\eqref{eq:RTvalue}.
We also have that 
\[
\binom{m+1}{(m+1)/2} 
= \biggl( 1-\Dfrac{1}{4m} + O(m^{-2}) \biggr) \Dfrac{2^{m+1}}{\sqrt{\pi(m+1)/2}}.
\]
Using Theorem \ref{T:cycle-clique} and routine asymptotic expansions, we obtain the first claim. The other two equalities follow from \eqref{def:Pauling}, \eqref{ours} and the second part of Lemma~\ref{L:tree-cycle-clique}.
\end{proof}


\nicebreak
\section{Other examples}\label{s:others}


\subsection{Triangular lattice $T_{n}$ and Baxter's constant}

For the triangular lattice $T_{n}$ of degree 6 on $n$ vertices,
with periodic boundary conditions,
Baxter~\cite{baxter1969f} proved in 1969 that 
\[
   \lim_{n\to\infty}\eo(T_n) = \log\dfrac{3 \sqrt 3}{2} \approx 0.95477.
\]
This compares poorly with Pauling's estimate 
$ \paul(T_n) = \log \frac52 \approx 0.91629$.
From \cite{glasser2005entropy}, we know that
the spanning tree entropy is
\[
 \lim_{n\to\infty}\tau(T_{n})  
= \Dfrac{5}{\pi} \sum_{i \ge 1} \Dfrac{ \sin(i \pi /3 )}{i^{2}}
\approx 1.61530,
\]
so our estimate \eqref{ours} gives 
\[
\ours(T_{n}) 
= 
\paul(T_n)
+ \dfrac{1}{2} \tau_6
- \dfrac{1}{2} \tau(T_{n}) 
\approx
0.95417,
\]
which is very close to the correct value.

\subsection{3-dimensional ice}\label{S:3Dice}

Of the several regular structures of water ice, we consider hexagonal ice (Ih) and cubic ice~(Ic).
Using a heuristic series expansion, Nagle~\cite{nagle1966} judged both $\eo(\mathrm{Ih})$ and $\eo(\mathrm{Ic})$ to lie in the interval [0.40992,0.41012].
By extrapolating finite simulations to the limit, Kolafa \cite{Kolafa2014} obtained the slightly higher value $\eo\approx 0.41043$ for both types of ice. Neither Nagle's nor Kolafa's calculations are sufficient to positively distinguish $\eo(\mathrm{Ih})$ from $\eo(\mathrm{Ic})$.

Using the method of Lyons~\cite{lyons2005} we find that the
spanning tree entropy of ice Ic is
\begin{align*}
 \tau(\mathrm{Ic}) &= \Dfrac18\int_0^1\!\!\int_0^1\!\!\int_0^1
    \log\bigl( 16464 - 3136(c_x+c_y+c_z) - 2016(c_xc_y+c_xc_z+c_yc_z) \\[-1ex]
   &\kern9em - 960c_xc_yc_z + 16(c_x^2c_y^2+c_x^2c_z^2+c_y^2c_z^2) \\
   &\kern9em -32(c_x^2c_yc_z +c_xc_y^2c_z +c_xc_yc_z^2)\bigr)
   \,dx\,dy\,dz \\
   &\approx 1.20645995,
\end{align*}
where $c_u$ means $\cos(2\pi u)$.
Nagle~\cite{nagle1966} noticed that the generating function for walks returning to the origin is the
same for both types of ice, implying that the eigenvalue distributions are the same and thus
$\tau(\mathrm{Ih})=\tau(\mathrm{Ic})$, which we verified to high precision.
Consequently, we have
\[
   \ours(\mathrm{Ih}) = \ours(\mathrm{Ic}) \approx 0.410433,
\]
in excellent agreement with Kolafa's estimate.

\subsection{Hypercubes $Q_{d}$}

The number of Eulerian orientations of a
$d$-dimensional hypercube $Q_{d}$ on $n = 2^{d}$ vertices
is only known up to $d=6$~\cite[sequence A358177]{OEIS}, and it appears that even the asymptotic value is unknown.

\begin{table}[ht!]
 \centering
 \begin{tabular}{|c|c|c|c|c|c|c|}
 \hline
   \vrule width0pt height2ex depth1ex
  $d$ & $n$ & $\paul$
    & $\ours$ & $\eo$ & $\paul-\eo$ & $\ours-\eo$\\
  \hline \hline
  4 & 16 & 0.405465 & 0.464780 & 0.499770 & -0.0943 & -0.035 \\
  6 & 64 & 0.916291 & 0.948381 & 0.955050 & -0.0388 & -0.0067 \\
  8 & 256 & 1.475907 & 1.489316 & 1.490759 & -0.0149 & -0.0014 \\
  10 & 1024 & 2.063693 & 2.069225 & 2.069554 & -0.0059 & -0.00033 \\
  12 & 4096 & 2.669829 & 2.672343 & 2.672420 & -0.0026 & -0.00008 \\
  14 & 16384 & 3.288868 & 3.290206 & 3.290224 & -0.0014 & -0.00002 \\
  \hline
 \end{tabular}
 \caption{Parameters for hypercubes $Q_d$\label{tab:hypercube}}
\end{table}

Using the method described in Section~\ref{s:numeric}, we have
computed estimates of $\eo(Q_d)$ up to $d=14$. From~\cite{bernardi2012spanning}, we  know that
\[
t( Q_{d} ) 
= \dfrac{1}{n} \prod_{i=1}^{d}\, (2i)_{}^{\binom{d}{i}}.
\]

The values shown for $\eo(Q_d)$ in Table~\ref{tab:hypercube}
are believed correct to within one value of the final digit.
It is seen that Pauling's estimate $\paul(Q_d)$ 
is improving as the dimension increases,
but our estimate $\ours(Q_d)$ is approaching the right answer more quickly.
Experimentally, it seems likely that $\eo(Q_d)=\ours(Q_d)+O(2^{-d})$.


 \section{Concluding remarks}

We conclude with a short summary of  interesting problems on the residual entropy of graphs mentioned in this paper that remain open.
\begin{enumerate}\itemsep=0pt
\item[(a)] Prove Conjecture \ref{Conj}.
\item[(b)] Give a combinatorial explanation for the strong
  correlation between residual entropy and spanning tree entropy.
  Theorem~\ref{T:IMZ} gives an analytic explanation for denser
  graphs. A qualitative hint is that the presence of short cycles
  tends to reduce the spanning tree count~\cite{mckay1983spanning}
  (at least for sparse graphs) but, by Theorem~\ref{EOGvalue}, 
  tends to increase the number of Eulerian orientations.
\item[(c)] Determine whether $\eo(\mathrm{Ih})$ and $\eo(\mathrm{Ic})$ coincide.
\item[(d)] Find the asymptotics of $\eo(Q_d)$ for hypercubes.
\item[(e)] Investigate the analogue of $\ours(G)$ for irregular graphs.
The average number of spanning trees for a wide range of degree
sequences was determined in~\cite{GIKM}.
\end{enumerate}

\end{document}